\documentclass[11pt,fleqn]{amsart}

\usepackage{amsmath, amssymb, amsthm}
\usepackage{mathrsfs}
\usepackage[utf8]{inputenc}
\usepackage[T1]{fontenc}
\usepackage[margin=1in]{geometry}
\usepackage{enumitem}

\usepackage{todonotes}
\usepackage{hyperref}

\DeclareMathOperator{\supp}{supp}
\DeclareMathOperator{\lip}{Lip}
\DeclareMathOperator{\LIP}{LIP}

\newtheorem{theorem}{Theorem}[section]
\newtheorem{lemma}[theorem]{Lemma}
\newtheorem{proposition}[theorem]{Proposition}
\newtheorem{problem}[theorem]{Problem}

\newtheorem{corollary}[theorem]{Corollary}

\theoremstyle{definition}

\theoremstyle{remark}
\newtheorem{remark}[theorem]{Remark}

\numberwithin{equation}{section}

\title[Continuous operators from spaces of Lipschitz functions]{Continuous operators from spaces of Lipschitz functions}
\author[C. Bargetz]{Christian Bargetz}
\address{Universität Innsbruck, Department of Mathematics, Innsbruck, Austria.}
\email{christian.bargetz@uibk.ac.at}
\author[J.\ K\k{a}kol]{Jerzy K\k{a}kol}
\address{Faculty of Mathematics and Computer Science, Adam Mickiewicz University, Pozna\'n, Poland, and Institute of Mathematics, Czech Academy of Sciences, Prague, Czech Republic.}
\email{kakol@amu.edu.pl}
\author[D.\ Sobota]{Damian Sobota}
\address{Kurt G\"odel Research Center, Department of Mathematics, Vienna University, Vienna, Austria.}
\email{damian.sobota@univie.ac.at}


\begin{document}

\begin{abstract}
  We study the existence of continuous (linear) operators from the Banach spaces $\lip_0(M)$ of Lipschitz functions on infinite metric spaces $M$ vanishing at a distinguished point and from their predual spaces $\mathcal{F}(M)$ onto certain Banach spaces, including $C(K)$-spaces and the spaces $c_0$ and $\ell_1$. For pairs of spaces $\lip_0(M)$ and $C(K)$ we prove that if they are endowed with topologies weaker than the norm topology, then usually no continuous (linear or not) surjection exists between those spaces. It is also showed that if a metric space $M$ contains a bilipschitz copy of the unit sphere $S_{c_0}$ of the space $c_0$, then $\lip_0(M)$ admits a continuous  operator onto $\ell_1$ and hence onto $c_0$. Using this, we provide several conditions for a space $M$ implying that $\lip_0(M)$ is not a Grothendieck space. Finally, we obtain a new characterization of the Schur property for Lipschitz-free spaces: a space $\mathcal{F}(M)$ has the Schur property if and only if for every complete discrete metric space $N$ with cardinality $d(M)$ the spaces $\mathcal{F}(M)$ and $\mathcal{F}(N)$ are weakly sequentially homeomorphic. 
\end{abstract}

\maketitle

\section{Introduction}

The goal of the present paper is to attempt a systematic study of the existence of continuous (linear) operators from the Banach space $\lip_0(M)$ of Lipschitz functions on a metric space $M$ vanishing at a distinguished point and from its predual space $\mathcal{F}(M)$, the Lipschitz-free Banach space of $M$, onto certain Banach spaces, including $C(K)$-spaces and the space $\ell_1$. This general line of  research is motivated by a series of results for specific metric spaces $M$, usually belonging to well-known classes of Banach spaces, see, e.g., the monograph~\cite{Weaver} and the articles \cite{Aliaga}, \cite{C-C-D}, \cite{LeandroGuzman}, \cite{Doucha}, \cite{GK}, or \cite{Hajek}.

The first main topic of this paper is closely related to the following general question.

\begin{problem}
  For which metric spaces $M$ and $N$, do there exist surjective continuous (not necessarily linear) mappings $(E,\tau)\rightarrow (F,\sigma)$, where $E\in \{C(M), \lip_0(M)\}$, $F\in \{C(N), \lip_0(N))\}$, and  $\tau,\sigma$ are any of the following topologies: norm topology, weak topology, compact-open topology, or pointwise topology?
\end{problem}

To deal with this problem, in Section \ref{sec:density} we compute the density characters of spaces of Lipschitz functions endowed with various topologies. For the norm topology, using Toru\'nczyk's classical theorem (\cite{Tor}), we observe in Corollary \ref{see} that, given two metric spaces $M$ and $N$, the Banach spaces $\lip_0(M)$ and $\lip_0(N)$ are homeomorphic if and only if $2^{d(M)}=2^{d(N)}$, where $d(X)$ denotes the density of a topological space $X$ (see Section \ref{sec:prelim} for relevant definitions and notations). The density of spaces $\lip_0(M)$ endowed with topologies significantly weaker than the norm topology or even the weak topology are studied in Theorem \ref{theorem:l_dens_tp_tk} and its several corollaries, in particular Theorem \ref{cor:lip_lip0_dens_tp_tk}. We show, among others, that if for a given metric space $M$ the space $\lip_0(M)$ is endowed with the topology $\tau$ contained between the pointwise topology and the compact-open topology, then for the densities we have $d(C(M),\tau)=d(\lip_0(M),\tau)\le d(M)$. This provides a metric extension for Noble's theorem \cite{Nob71}, cf. also \cite{GFLP}, as well as relates to results presented in the recent paper \cite{C-C-V} which concern the density character of the weak* topology on spaces of Lipschitz functions.

\medskip

Section \ref{sec:operators} significantly refers to the title of this work. The  motivating result for the section is Corollary~\ref{cor:lip0m_ck} stating that for an infinite metric space $M$ there always exist a compact space $K$ with density $d(K)=d(M)$ and weight $w(K)=2^{d(M)}$, and a continuous linear surjection from the Banach space $\lip_0(M)$ onto the Banach space $C(K)$. It appears that if we weaken the norm topology of $\lip_0(M)$ to the pointwise topology, then such a surjection does not exist. Namely, in Theorem~\ref{theorem_lipp_onto_lipw} we prove that if $\lip_0(M)$ is endowed with the pointwise topology $\tau_p$, then there is no continuous (linear or not) surjection from the space $(\lip_0(M),\tau_p)$ onto any space of the form $(\lip_0(N),\tau)$ or $(C(X),\tau')$ with Tychonoff $X$, where the topologies $\tau$ and $\tau'$ contain, respectively, the weak topology of the Banach space $\lip_0(N)$ and the pointwise topology of $C(X)$. The latter in particular applies to the case of a compact space $X$ and the supremum norm topology of the Banach space $C(X)$.

In Theorem \ref{theorem:cpm_lip0n} we consider the reverse situation, which also appears to be negative. Namely, we show that for every infinite Tychonoff space $X$ and infinite metric space $N$ for which the inequality $w(X)<2^{d(N)}$ holds, there is no continuous (linear or not) surjection from $(C(X),\tau)$ onto $(\lip_0(N),\tau')$, where $\tau$ is a topology contained between the pointwise topology and the compact-open topology on $C(X)$ and a topology $\tau'$ contains the weak topology of the Banach space $\lip_0(N)$. The assumption $w(X)<2^{d(N)}$ cannot be removed nor even relaxed to $w(X)\le2^{d(N)}$, which is justified by the well-known isomorphisms $C(\beta\mathbb{N})\simeq\ell_\infty\simeq\lip_0([0,1])$, where $\beta\mathbb{N}$ is the \v{C}ech--Stone compactification of the set $\mathbb{N}$ of natural numbers. 
On the other hand, it should be pointed out, that  according to C\'{u}th, Doucha, and Wojtaszczyk, see~\cite[proof of Theorem 4.1]{Doucha}, there is no continuous linear surjection $T\colon C(K)\rightarrow\lip_0([0,1]^2)$ for any compact space $K$.  This in particular yields that for any metric space $M$ containing the square $[0,1]^2$ there is no continuous linear surjection $T\colon C(K)\rightarrow\lip_0(M)$ for any compact space $K$.

\medskip

In Subsection \ref{sec:lipschitz_free} we study the existence of continuous mappings from or onto Lipschitz-free Banach spaces. In Proposition \ref{prop:fm_lip0n} we observe that, again by virtue of the Toru\'nczyk theorem, for an infinite metric space $M$  and a Banach space $E$ there exists a continuous (linear) surjection from the Lipschitz-free space $\mathcal{F}(M)$ onto $E$ if and only if $\mathcal{F}(M)$ contains a subset homeomorphic to $E$ if and only if $d(M)\ge d(E)$. Further, motivated by a recently studied open problem due to Krupski \cite{Krupski-1} (see also \cite{K-M}), asking whether the  spaces  $C_p(K)$ and $C(L)_w$ can be homeomorphic for infinite compact spaces $K$ and $L$, where, for a Tychonoff space $X$ and a Banach space $E$, $C_p(X)$ denotes the space $C(X)$ endowed with the pointwise topology and $E_w$ denotes $E$ endowed with the weak topology, we discuss in Theorem~\ref{theorem_cpm_homeo_fmw} possible natural conditions put on two metric spaces $M$ and $N$ under which there is no homeomorphism between spaces $C_p(M)$ and  $\mathcal{F}(N)_w$.

One of the conditions considered in Theorem \ref{theorem_cpm_homeo_fmw} is that the space $\mathcal{F}(N)$ has the Schur property. Recall that the a Banach space $E$ has the \emph{Schur property} if every weakly convergent sequence in $E$ is norm convergent. It is a classical result that the Banach space $\ell_1$ has the Schur property. The Schur property of Lipschitz-free spaces has been thoroughly studied by Aliaga \textit{et al.} in \cite{AGPP} and~\cite{Aliaga}, where it was proved, among others, that, for a given metric space $M$, the space $\mathcal{F}(M)$ has the Schur property if and only if the Banach space $L_1([0,1])$ does not isomorphically embed into $\mathcal{F}(M)$ if and only if the completion $\tilde{M}$ of $M$ is purely $1$-unrectifiable, that is, $\tilde{M}$ does not contain bilipschitz copies of subsets of the real line $\mathbb{R}$ of non-zero Lebesgue measure.

We provide a new characterization of the Schur property of Lipschitz-free spaces, more in the spirit of the title of our paper. Namely, in Theorem \ref{theorem_schur_discrete} we prove that for an infinite  metric space  $M$  the space $\mathcal{F}(M)$ has the Schur property if and only if for every complete discrete metric space $N$ with cardinality $d(M)$ the spaces $\mathcal{F}(M)_w$ and $\mathcal{F}(N)_w$ are sequentially homeomorphic. The main tool used for proving this result is Theorem \ref{lee} stating that a Banach space $E$ has the Schur property if and only if the spaces $E$ and $E_w$ are sequentially homeomorphic. We also observe in Theorem \ref{theorem:weak_seq_homeo} that, for every two Banach spaces $E$ and $F$ both having the Schur property, the spaces $E_w$ and $F_w$ are sequentially homeomorphic if and only if $d(E)=d(F)$.

\medskip

Section \ref{sec:ell_1} is devoted to the second main problem of this paper, motivated by the fact that for every infinite metric space $M$, the space $\lip_0(M)$ contains a complemented copy of the Banach space $\ell_\infty(d(M))$ (by H\'{a}jek--Novotn\'{y}'s result \cite{Hajek}, see also Theorem \ref{cor1}), hence there is always a continuous linear surjection $T\colon\lip_0(M)\rightarrow\ell_2(2^{d(M)})$ and thus a continuous linear surjection $S\colon\lip_0(M)\rightarrow\ell_2$, see Corollary~\ref{separable}.

\begin{problem}
For which metric spaces $M$, in particular for which Banach spaces $M$, do there exist surjective continuous operators from the Banach spaces $\lip_0(M)$ onto the space $\ell_1$?
\end{problem}

The main result of Section \ref{sec:ell_1}, Theorem~\ref{thm:BiLipCK}, asserts that if $E$ is a separable Banach space which is an absolute Lipschitz retract and contains an isomorphic copy of the Banach space $c_0$, then for every metric space $M$ containing a bilipschitz copy of the unit sphere $S_{E}$ of $E$ the space $\lip_0(M)$ admits a continuous  operator onto $\ell_1$. Hence, the conclusion holds in particular
for any metric space $M$ containing a bilipschitz copy of the unit sphere $S_{c_0}$ of $c_0$, see Corollary~\ref{cor:bilipschitz_c0}.

Using Corollary \ref{cor:bilipschitz_c0} and the theorem of Dalet \cite{Dalet2} stating that the space $\lip_0(\ell_1)$ contains a complemented copy of $\ell_1$, we provide in Theorem \ref{theorem:examples_onto_ell1} an extensive list of examples of Banach spaces $E$ for which the spaces $\lip_0(E)$ admit continuous  operators onto $\ell_1$. To this list belong, e.g., the following Banach spaces:
\begin{itemize}
\item $C(K)$-spaces and so $L_\infty(\mu)$-spaces, 
\item $L_1(\mu)$-spaces,
\item $\lip_0(M)$-spaces,
\item $\mathcal{F}(M)$-spaces.
\end{itemize}

Containing bilipschitz copies of the unit sphere of an infinite-dimensional Banach space is not necessary for a metric space $M$ to admit a continuous  operator from its Banach space $\lip_0(M)$ onto $\ell_1$, as was already observed in \cite[Remark 3.6]{C-C-D} for nets in $c_0$. Nevertheless, we show that there exists an infinite-dimensional arcwise-connected $\sigma$-compact metric space $M$ not containing the unit ball of any infinite-dimensional normed space but such that $\lip_0(M)$ admits a continuous operator onto $\ell_1$, see Proposition~\ref{prop:NoGrothendieckNoc0Ball}.

\medskip

The last research part of the paper, provided in Subsection \ref{sec:grothendieck}, concerns the following problem.

\begin{problem}
For which metric spaces $M$, in particular for which Banach spaces $M$, do the spaces $\lip_0(M)$ have the Grothendieck property?
\end{problem}

Recall that a Banach space $E$ is called \emph{Grothendieck} (or is said to have the \emph{Grothendieck property}) if every weak* convergent sequence in the dual space $E^*$  converges weakly. Typical examples of Grothendieck spaces are reflexive spaces, the space $\ell_\infty$ or more generally spaces $C(K)$ for $K$ extremally disconnected (Grothendieck \cite{Gro53}), the space $H^\infty$ of all bounded analytic functions on the unit disk (Bourgain \cite{Bou83}), and von Neumann algebras (Pfitzner \cite{Pfi94}). On the other hand, the spaces $\ell_1$ and $c_0$, or more generally spaces $C(K)$ for $K$ metric, are not Grothendieck. We refer the reader to \cite{Kania} for an extensive survey on Grothendieck Banach spaces. Notice that no space $\mathcal{F}(M)$  is a Grothendieck space since it contains a complemented copy of $\ell_{1}$ (see Corollary~\ref{hano_ell1}) and the Grothendieck property is preserved by continuous linear surjections.

It seems that apart from $\lip_0([0,1])\simeq\ell_\infty$ there is no known example of a Banach space  $\lip_0(M)$ which is a Grothendieck space. Recall that for the class of $C(K)$-spaces we have the following result, due to Cembranos \cite{Cem84} and R\"{a}biger \cite{Rab}: for every compact space $K$, the Banach space $C(K)$ is Grothendieck if and only if $C(K)$ does not contain any complemented copy of  $c_0$. This characterization cannot work for spaces $\lip_0(M)$, since for each metric space $M$ the space $\lip_0(M)$ is isometrically isomorphic to the dual Banach space $\mathcal{F}(M)^*$, hence it does not  contain any complemented copy of $c_0$ (see e.g. \cite[Theorem 2.4.15]{DDLS}), even though, by Theorem~\ref{cor1}, $\lip_0(M)$ always contains \textit{some} copy of $c_{0}$. It follows that to decide that a given space $\lip_0(M)$ is not Grothendieck we have to look for continuous linear surjections from $\lip_0(M)$ onto $c_0$ or prove that the dual space $\lip_0(M)^*$ is not weakly sequentially complete; see the discussion at the beginning of Section \ref{sec:grothendieck} for more details.

Based on results from the beginning of Section \ref{sec:ell_1}, in Subsection \ref{sec:grothendieck} we provide a number of conditions for metric spaces $M$ implying that their spaces $\lip_0(M)$ are not Grothendieck. For example, since the Grothendieck property is preserved by continuous linear surjections and the space  $\lip_0(E)$ contains a complemented copy of the dual $E^*$ for each  Banach space $E$, the space  $\lip_0(E)$ is not  Grothendieck provided that $E^*$ is not Grothendieck. This observation yields a class of Banach spaces $E$ for which $\lip_0(E)$ is not Grothendieck, e.g. all non-reflexive spaces with separable duals belong to this class, see Proposition \ref{prop:dual_not_groth}. 

An immediate consequence of Theorem~\ref{thm:BiLipCK} is that if $E$ is a separable Banach space which is an absolute Lipschitz retract and contains an isomorphic copy of the Banach space $c_0$, then for every metric space $M$ containing a bilipschitz copy of the unit sphere $S_{E}$ of $E$ the space $\lip_0(M)$ is not Grothendieck. Thus, if a metric space $M$ contains a bilipschitz copy of the unit sphere $S_{c_0}$ of $c_0$, then $\lip_0(M)$ is not Grothendieck; see Theorem \ref{Fo} and Corollary \ref{Gro}. Moreover, using the aforementioned Theorem \ref{theorem:examples_onto_ell1}, we get that if a Banach space $E$ is a $C(K)$-space, $L_1(\mu)$-space, $\lip_0(M)$-space, or $\mathcal{F}(M)$-space, then $\lip_0(E)$ is also not Grothendieck, see Corollary \ref{cor:many_non_grothendieck_spaces}.

\medskip

We complete the paper by providing several open questions in Section \ref{sec:problems}.

\section{Preliminaries\label{sec:prelim}}

Let $M$ be a metric space. If not stated otherwise, $\rho$ always denotes \textit{the} metric of $M$. For $x\in M$ and $\emptyset\neq A\subseteq M$, we write $\rho(x,A)=\inf\{\rho(x,y)\colon y\in A\}$. If $a\in M$ and $r>0$, then $B(a,r)=\{y\in M\colon \rho(a,y)<r\}$ is the open ball with center $a$ and radius $r$.

By $\LIP(M)$ we denote the vector space of all Lipschitz real-valued functions on $M$. By default, no topology is declared on $\LIP(M)$. As usual, by $\lip(M)$ we denote the Banach space of all bounded functions $f\in\LIP(M)$, endowed with the norm $\|f\|_{\lip(M)}=\|f\|_\infty+\lip(f)$, where $\|f\|_\infty$ denotes the supremum norm of $f$, i.e. $\|f\|_\infty=\sup_{x\in M}|f(x)|$, and $\lip(f)$ denotes the \emph{Lipschitz constant} of $f$, i.e.
\[\lip(f)=\sup_{\substack{x,y\in M\\x\neq y}}\frac{|f(x)-f(y)|}{\rho(x,y)}.\]
For a distinguished point $e\in M$, called the \emph{base point} of $M$, $\lip_0(M)$ denotes the Banach space of all functions $f\in\LIP(M)$ such that $f(e)=0$, equipped with the norm $\|f\|_{\lip_0(M)}=\lip(f)$. Note that for different choices of base points we get spaces $\lip_0(M)$ which are isometrically isomorphic. $\mathcal{F}(M)$ denotes the \emph{Lipschitz-free space} of $M$, that is, the Banach space equal to the closure of the linear subspace of the dual space $\lip_0(M)^*$ spanned by all one-point measures on $M$, i.e. $\mathcal{F}(M)=\overline{\mbox{span}\{\delta_x\colon x\in M\}}^{\|\cdot\|_{\lip_0(M)^*}}$. Recall that the dual space $\mathcal{F}(M)^*$ is isometrically isomorphic to $\lip_0(M)$.

\medskip

All topological spaces considered in this paper are assumed to be \emph{Tychonoff}, that is, completely regular and Hausdorff. For a topological space $X$, by $C(X)$ we denote the set of all continuous real-valued functions on $X$. $C_p(X)$ denotes $C(X)$ endowed with the pointwise topology $\tau_p$ (inherited from $\mathbb{R}^X$) and $C_k(X)$ denotes $C(X)$ endowed with the compact-open topology $\tau_k$. Note that $\tau_p\subseteq\tau_k$. If $\tau$ is a topology on $C(X)$ (e.g. $\tau=\tau_p$ or $\tau=\tau_k$) and $E$ is a subspace of $(C(X),\tau)$, then the subspace topology of $E$ will still be denoted by $\tau$.

If $X$ is a compact (Hausdorff) space, then, as usual, we endow $C(X)$ with the supremum norm $\|\cdot\|_\infty$. If $X$ is locally compact, then $C_0(X)$ denotes the subspace of $C(X)$ consisting of all functions which vanish at infinity, also endowed with the norm $\|\cdot\|_\infty$. In both cases the resulting spaces are Banach spaces.

For a metric space $M$, the spaces $\LIP(M)$, $\lip(M)$, and $\lip_0(M)$ are obviously subspaces of $C(M)$. Let $\lip_{0}(M)_p$ denote the space $\lip_{0}(M)$ endowed with the pointwise topology $\tau_p$ (inherited from $(C(M),\tau_p)$).

\medskip

Unless stated otherwise all vector spaces considered in this article are assumed to be over the field of real numbers. For a Banach space $E$, the closed unit ball of $E$ will be denoted by $B_E$. $E^*$ denotes the dual Banach space of $E$ and $E_w$ denotes $E$ endowed with the weak topology, i.e. the topology $w=\sigma(E,E^*)$. 
The weak* topology of $E^*$ will be denoted by $w^*$, i.e. $w^*=\sigma(E^*,E)$. 

If two Banach spaces $E$ and $F$ are isomorphic, then it is denoted by $E\simeq F$. If $G$ and $H$ are topological vector spaces (in particular, Banach spaces) and $T\colon G\to H$ is a mapping, then we do not \textit{a priori} assume that $T$ is linear; if $T$ is linear, then we call it an \emph{operator}.

The Banach spaces $c_0$ and $L_p(\Omega,\Sigma,\mu)$ ($1\le p\le\infty$), in particular $\ell_1(\Gamma)$ and $\ell_\infty(\Gamma)$ for a set $\Gamma$ and $\ell_\infty^n$ for an integer $n\in\mathbb{N}$, are defined in the standard way. Recall that a Banach space $E$ is \emph{$\ell_p$-saturated} for some $1\le p<\infty$ if every infinite-dimensional closed linear subspace of $E$ contains an isomorphic copy of the space $\ell_p$. 

\medskip

For a topological space $X$, by $w(x)$ we denote the \emph{weight} of $X$ (i.e. the minimal cardinality of a base of the topology of $X$), by $d(X)$ the \emph{density} of $X$ (i.e. the minimal cardinality of a dense subset of $X$), by $d_{seq}(X)$ the \emph{sequential density} of $X$ (i.e. the minimal cardinality of a \emph{sequentially dense} subset of $X$, that is, such a subset $A\subseteq X$ that for every $x\in X$ there exists a sequence $(x_n)_{n=0}^\infty$ in $A$ converging to $x$), by $c(X)$ the \emph{cellularity} of $X$ (i.e. the supremum of cardinalities of families of pairwise disjoint non-empty open subsets of $X$), and by $ww(X)$ the \emph{weak weight} of $X$ (i.e. the minimal weight $w(Y)$ of a Tychonoff space $Y$ such that there exists a continuous bijection $\varphi\colon X\rightarrow Y$). We always have $c(X)\le d(X)\le w(X)$, $d(X)\le d_{seq}(X)$, and $w(X)\ge ww(X)$. If $X$ is metric, then we have $c(X)=d(X)=d_{seq}(X)=w(X)\ge ww(X)$, and the cellularity is attained, that is, there exists a collection of pairwise disjoint non-empty open subsets of $X$ of cardinality $c(X)$ (see \cite[Theorem 8.1]{Hodel}).

\section{Density of spaces of Lipschitz functions\label{sec:density}}

In this paper we will frequently use the fact that for every infinite metric space $M$ the spaces $\lip(M)$ and $\lip_0(M)$ contain isomorphic copies of the Banach space $\ell_\infty(d(M))$, where $d(M)$ denotes the density of $M$ (see Theorems \ref{thm:linfLip} and \ref{cor1}). These results can be easily deduced from the proof of H\'ajek and Novotn\'{y}'s \cite[Proposition 3]{Hajek} when combined with Rosenthal's \cite[Corollary 1.2]{Rosenthal}. Since our proof is somewhat more direct, we add it for the convenience of the reader and to keep this paper more self-contained. The main difference between our proof and the one in~\cite{Hajek} is that we use the cellularity of $M$ to obtain directly a suitable family of disjoint open balls, which makes the argument simpler.

\begin{theorem}\label{thm:linfLip}
  For every infinite metric space $M$ the Banach space $\lip(M)$ contains an isomorphic copy of $\ell_{\infty}(d(M))$.
\end{theorem}

\begin{proof}
  Since $M$ is a metric space, its cellularity $c(M)$ is equal to the density $d(M)$, and so we may pick a family of pairwise disjoint (open) balls $(B(y_\alpha , r_\alpha ))_{\alpha< d(M)}$ of cardinality $d(M)$ such that $r_\alpha\le 1$ for every $\alpha<d(M)$ (see \cite[Theorem~8.1]{Hodel}). For each $\alpha<d(M)$ we set
  \[
    \delta_{\alpha} = \rho(y_\alpha, M\setminus B(y_\alpha,r_\alpha/2)) = \inf\big\{\rho(x,y_\alpha)\colon x \in  M\setminus B(y_\alpha,r_\alpha/2)\big\}
  \]
  and, for every $x\in M$,
  \[
    f_\alpha(x) = \max\big\{\min\{\delta_\alpha,1\}-\rho(x,y_\alpha),\ 0\big\}.
  \]
  Note that $\min\{1,\delta_\alpha\}\ge r_\alpha/2$ and that $f_\alpha\colon M\rightarrow\mathbb{R}$ is such a bounded continuous function that
  \[
    \supp(f_\alpha)\subseteq\overline{B(y_\alpha,r_\alpha/2)}\subseteq B(y_\alpha,r_\alpha)\quad\text{and}\quad f_\alpha^{-1}[(0,\infty)]\subseteq B(y_\alpha,r_\alpha/2)
  \]
  as well as $\|f_\alpha\|_\infty=f_\alpha(y_\alpha)=\min\{1,\delta_\alpha\}$, so $\|f_\alpha\|_\infty\le1$.

  Fix $\alpha<d(M)$. We will estimate the Lipschitz constant of $f_\alpha$. Since $f_\alpha$ vanishes outside the ball $B(y_\alpha,r_\alpha/2)$, we only have to consider two cases. For $x,y\in B(y_\alpha,r_\alpha/2)$ we have
  \[
    |f_\alpha(x)-f_\alpha(y)| = |\rho(x,y_\alpha)-\rho(y,y_\alpha)| \le \rho(x,y).
  \]
  For $x\in B(y_\alpha,r_\alpha/2)$ and $y\not\in B(y_\alpha,r_\alpha/2)$ we have
  \[
    |f_\alpha(x)-f_\alpha(y)| = |f_\alpha(x)| \le\min\{\delta_\alpha,1\} - \rho(x,y_\alpha) \le \rho(y,y_\alpha)-\rho(x, y_\alpha) \le \rho(x,y).
  \]
  From these two inequalities we conclude that $\lip(f_\alpha)\le 1$. By the definition of $\delta_\alpha$ we may pick a sequence $(z_n)_{n=0}^{\infty}$ in $M\setminus B(y_\alpha,r_\alpha/2)$ with $\rho(y_\alpha,z_n)\searrow \delta_\alpha$ for $n\rightarrow \infty$. Using this sequence we have
  \[
    \frac{|f_\alpha(y_\alpha)-f_\alpha(z_n)|}{\rho(y_\alpha,z_n)} = \frac{\min\{1,\delta_\alpha\}}{\rho(y_\alpha,z_n)} \nearrow \min\{1,1/\delta_\alpha\},
  \]
  which implies that
  \[\tag{$*$}1\ge\lip(f_\alpha)\ge \min\{1,1/\delta_\alpha\},\]
  and hence $f_\alpha\in\lip(M)$.

  Since the functions $f_\alpha$'s have disjoint supports, we may define the operator
  \[
    T\colon \ell_\infty(d(M)) \rightarrow \lip(M), \qquad v=(v_\alpha)_{\alpha<d(M)} \longmapsto T(v)= \sum_{\alpha<d(M)} v_\alpha f_\alpha.
  \]

  Fix $v\in\ell_\infty(d(M))$. It is immediate that $\|T(v)\|_\infty \le \|v\|_\infty$, so $T(v)$ is a bounded function. To compute an upper bound for the Lipschitz constant of $T(v)$, we only need to consider the case of $x\in \supp(f_\alpha)\subseteq \overline{B(y_\alpha,r_\alpha/2)}$ and $y\in \supp(f_\beta)\subseteq \overline{B(y_\beta, r_\beta/2)}$, for $\alpha\neq\beta<d(M)$, since the other cases are already covered by the computation of the Lipschitz constants of $f_\alpha$'s (see ($*$)). We have
  \begin{align*}
    |T(v)(x)-T(v)(y)| &= |v_\alpha f_\alpha(x)- v_\beta f_\beta(y)| \le\\
    &\le|v_\alpha| (\min(1,\delta_\alpha)-\rho(x,y_\alpha)) + |v_\beta| (\min(1,\delta_\beta)-\rho(y,y_\beta))\le\\
    &\le |v_\alpha|(\rho(y,y_\alpha)-\rho(x,y_\alpha)) + |v_\beta| (\rho(x,y_\beta)-\rho(y,y_\beta)) \le 2\|v\|_\infty \rho(x,y).
  \end{align*}
Hence, $\lip(T(v))\le2\|v\|_\infty$, and so $T(v)\in\lip(M)$. We conclude that
  \[
    \|T(v)\|_{\lip(M)}= \|T(v)\|_\infty + \lip(T(v))\le 3 \|v\|_\infty.
  \]
  Since for all $\alpha<d(M)$, by ($*$), we also have
  \[
    \|T(e_\alpha)\|_{\lip(M)}=\|f_\alpha\|_{\lip(M)} = \|f_\alpha\|_\infty + \lip(f_\alpha)\ge\min\{1,\delta_\alpha\}+\min\{1,1/\delta_\alpha\}\ge 1,
  \]
  where $e_\alpha$ denotes the $\alpha$-th standard basic vector of $\ell_\infty(d(M))$, we may use Rosenthal's \cite[Theorem~7.10]{HajekBi} to conclude the existence of a set $\Gamma$ with cardinality $d(M)$ such that $\ell_\infty(\Gamma)$ is isomorphic to a subspace of $\lip(M)$. This finishes the proof.
\end{proof}

Following \cite[Definition~2.17]{Weaver}, for a complete metric space $M$ with the metric $\rho$ and the base point $e$, the space $M^\dagger$ is defined as $M\setminus\{e\}$ and equipped with the metric
\[
  \rho^\dagger(x,y) = \frac{\rho(x,y)+|\rho(x,e)-\rho(y,e)|}{\max\{\rho(x,e),\rho(y,e)\}}
\]
for $x,y\in M\setminus\{e\}$. By \cite[Theorem 2.20]{Weaver}, the Banach spaces $\lip_0(M)$ and $\lip(M^\dagger)$ are isomorphic---we will use this result together with Theorem \ref{thm:linfLip} to deduce the following theorem of H\'{a}jek and Novotn\'{y} \cite{Hajek}, which will constitute one of our main tools in what follows.

\begin{theorem}[H\'{a}jek--Novotn\'{y}]\label{cor1}
  Let $M$ be an infinite  metric space. The Banach space $\lip_0(M)$ contains an isomorphic copy of $\ell_{\infty}(d(M))$.
\end{theorem}

\begin{proof}
Since for every metric space $M$ and its completion $\tilde{M}$ the spaces $\lip_0(M)$ and $\lip_0(\tilde{M})$ are isometrically isomorphic, we may without loss of generality assume that $M$ is complete. We show that the identity mapping $f\colon M^\dagger\rightarrow M\setminus\{e\}$ is a homeomorphism, which implies that $d(M)=d(M^\dagger)$. Let $(x_n)_{n=0}^\infty$ be a sequence in $M\setminus\{e\}$.

  First note that if $(x_n)_{n=0}^\infty$ converges to some $x\in M\setminus\{e\}$ with respect to $\rho$, then we trivially have $\rho^\dagger(x_n,x)\rightarrow 0$ as well, which implies that $f$ is sequentially continuous and hence continuous.

  Now assume that $(x_n)_{n=0}^\infty$ converges to $x\in M\setminus\{e\}$ with respect to $\rho^\dagger$. 
  Note that for every $n\in\mathbb{N}$ we have
  \[\max\{\rho(x_n,e),\rho(x,e)\}\le\max\{\rho(x_n,x)+\rho(x,e),\rho(x,e)\}=\rho(x_n,x)+\rho(x,e),\]
 hence
  \[
     \rho^\dagger(x_n,x) = \frac{\rho(x_n,x)+|\rho(x_n,e)-\rho(x,e)|}{\max\{\rho(x_n,e),\rho(x,e)\}}\ge
     \frac{\rho(x_n,x)}{\rho(x_n,x)+\rho(x,e)},
  \]
  which, after the rearrangement, gives
  \[
  \rho^\dagger(x_n,x)\ge\rho(x_n,x)\cdot\frac{1-\rho^\dagger(x_n,x)}{\rho(x,e)}.
  \]
  The latter inequality yields that $\rho(x_n,x)\rightarrow0$ as $\rho^\dagger(x_n,x)\to0$. It follows that the inverse mapping $f^{-1}$ is also sequentially continuous and hence continuous. Consequently, $f$ is a homeomorphism between $M\setminus\{e\}$ and $M^\dagger$, and so $d(M)=d(M^\dagger)$. By aforementioned \cite[Theorem~2.20]{Weaver}, the Banach spaces  $\lip_0(M)$ and $\lip(M^\dagger)$ are isomorphic and hence the claim follows from Theorem~\ref{thm:linfLip}.
\end{proof}

As a direct consequence of Theorem \ref{cor1} we obtain H\'{a}jek and Novotný's (\cite[Proposition 3]{Hajek}). 

\begin{corollary}[H\'{a}jek--Novotný]\label{hano_ell1}
  For every infinite metric space $M$, the space $\mathcal{F}(M)$ contains a complemented copy of $\ell_{1}(d(M))$.
\end{corollary}

\begin{proof}
  Since by Theorem~\ref{cor1} the dual space $\lip_0(M)$ of $\mathcal{F}(M)$ contains a copy of $\ell_\infty(d(M))$, the result follows from \cite[Corollary~1.2]{Rosenthal}.
\end{proof}

For the next corollary recall that $d(\ell_\infty(\Gamma))=2^{|\Gamma|}$ for any infinite set $\Gamma$.

\begin{corollary}\label{To}
For every infinite metric space $M$, the Banach spaces $\lip_0(M)$ and $\lip(M)$ both have density $2^{d(M)}$.
\end{corollary}
\begin{proof}
  We already know that $d(\lip_0(M))\ge 2^{d(M)}$ and $d(\lip(M))\ge 2^{d(M)}$. The other direction follows from the fact that there are at most $2^{d(M)}$ many continuous real-valued functions on $M$, see e.g. \cite[Theorem~10.1]{Hodel}.
\end{proof}

Corollary \ref{To}, together with the classical Toru\'nczyk theorem (\cite{Tor}) stating that two infinite-dimensional Banach spaces are homeomorphic if and only if they have the same density, yields the following result. 

\begin{corollary}\label{see}
  Let $M$ and $N$ be infinite metric spaces. The Banach spaces $\ell_\infty(d(M))$, $\lip_0(M)$, $\lip(M)$, $\lip_0(N)$, and $\lip(N)$ are all pairwise homeomorphic if and only if $2^{d(M)}=2^{d(N)}$.
\end{corollary}

Of course, if $d(M)=d(N)$, then $2^{d(M)}=2^{d(N)}$, however the converse implication in general does not hold, as there are models of set theory (like original Cohen's model for violating the Continuum Hypothesis, see~\cite[Chapter VII]{Kunen}) in which the equality $2^{\aleph_0}=2^{\aleph_1}$ is satisfied.

Let us now move to other topologies on spaces of Lipschitz functions. We will be in particular interested in topologies laying between the pointwise topology $\tau_p$ and the compact-open topology $\tau_k$.

Let $(c_{0})_p$ be the standard Banach space $c_0$ but endowed with the product topology $\tau_p$ inherited from $\mathbb{R}^{\mathbb{N}}$. For every infinite Tychonoff space $X$ the space $C_p(X)$ contains a non-necessarily closed copy of $(c_0)_p$; a closed copy can be found provided that $X$ contains an infinite compact subset, see \cite{KMS} for a detailed discussion. The question when $C_p(X)$ contains a complemented copy of $(c_0)_p$ was investigated in \cite{BKS1} (cf. also \cite{KMSZ,KSZ,KSZ2}). Note that, by virtue of the Closed Graph Theorem, if for a compact space $X$ the space $C_p(X)$ contains a complemented copy of  $(c_{0})_p$, then the Banach space $C(X)$ contains a complemented copy of $c_0$ (the converse is however not true, see \cite{KSZ2} for an appropriate counterexample).
\begin{proposition}\label{prop_lip0p_c0p}
For any metric space $M$, the space $\lip_{0}(M)_p$ does not contain any closed copy of  $(c_{0})_p$.
 \end{proposition}
\begin{proof}
Assume that $E$ is a closed linear subspace of $\lip_0(M)_p$ isomorphic to $(c_{0})_p$. Since the closed unit ball $B$ of $\lip_{0}(M)$ is compact in $\lip_{0}(M)_p$, for the sequence $(B_{n})_{n=0}^\infty$ of compact sets, where $B_{n}=nB$ for $n\in\mathbb{N}$, we obtain the compact cover $(B_{n}\cap E)_{n=0}^\infty$ of $E$. On the other hand, the space $(c_0)_p$ is isomorphic to the space $C_p(\mathbb{N}^{\#})$, where $\mathbb{N}^{\#}$ is a one-point compactification of the discrete space $\mathbb{N}$ of natural numbers. It follows that $C_p(\mathbb{N}^{\#})$ is covered by a sequence of compact sets. This however yields a contradiction by \cite[Theorem I.2.1]{Arch}.
\end{proof}
\begin{remark}
For any infinite non-discrete metric space $M$ its space $C_p(M)$ contains a complemented (so closed) copy of $(c_0)_p$ (see \cite{BKS1}). It follows from Proposition~\ref{prop_lip0p_c0p} that this copy of $(c_0)_p$ always contains a non-Lipschitz function.
\end{remark}

We will now compute densities of spaces of Lipschitz functions endowed with various topologies weaker than the compact-open topology. The next theorem is an extension of Noble's \cite[Theorem 1]{Nob71}.

\begin{theorem}\label{theorem:l_dens_tp_tk}
Let $M$ be a metric space. Let $\tau$ be any topology on the space $C(M)$ laying between the pointwise topology and the compact-open topology, i.e., $\tau_p\subseteq\tau\subseteq\tau_k$. Then,
\begin{enumerate}
	\item the space $\LIP(M)$ is dense in $(C(M),\tau)$, and
	\item $ww(M)=d(C(M),\tau)=d(\LIP(M),\tau)\le d(M)$.
\end{enumerate}
\end{theorem}
\begin{proof}
Since $M$ is metric, $d(M)=w(M)\ge ww(M)$. By \cite[Theorem 1]{Nob71}, $ww(M)=d(C(M),\tau)$.

We first show that $ww(M)\ge d(\LIP(M),\tau)$. Since $d(\LIP(M),\tau)\le d(\LIP(M),\tau_k)$, we may assume that $\tau=\tau_k$. Let $Y$ be a Tychonoff space such that there exists a continuous bijection $\varphi\colon M\rightarrow Y$ and a base $\mathcal{B}$ of the topology of $Y$ such that $|\mathcal{B}|=ww(M)$.

Let $U,V\in\mathcal{B}$ be a pair of non-empty sets such that $\overline{U}\cap\overline{V}=\emptyset$ and put $G^{U,V}=\varphi^{-1}[U]$ and $H^{U,V}=\varphi^{-1}[V]$. Of course, $\overline{G^{U,V}}\cap\overline{H^{U,V}}=\emptyset$. For each $n>0$ put also
\[G^{U,V}_n=\bigcup\Big\{B\big(x,1/(3n^2)\big)\colon\ x\in G^{U,V},\ \rho(x,H^{U,V})\ge1/n\Big\}\cap G^{U,V}.\]
We have $G^{U,V}_n\subseteq G^{U,V}_{n+1}\subseteq G^{U,V}$. Also, if $G^{U,V}_n\neq\emptyset$, then $\rho(x,G^{U,V}_n)\ge1/n-{1}/{(3n^2)}$ for every $x\in H^{U,V}$.

For each $n>0$ such that $G^{U,V}_n\neq\emptyset$, define a Urysohn function $f^{U,V}_n\colon M\rightarrow[0,1]$ for the sets $\overline{G^{U,V}_n}$ and $\overline{H^{U,V}}$ in the standard way:
\[f^{U,V}_n(x)=\frac{\rho(x,G^{U,V}_n)}{\rho(x,G^{U,V}_n)+\rho(x,H^{U,V})}\]
for all $x\in M$. Also, for each $n>0$ such that $G^{U,V}_n=\emptyset$, simply set $f^{U,V}_n(x)=0$ for all $x\in M$.
Note that in either case $f^{U,V}_n$ is Lipschitz (see e.g.~\cite[Proposition~2.1.1]{CMN}).

Since $\overline{G^{U,V}}\cap\overline{H^{U,V}}=\emptyset$, for every $x\in G^{U,V}$ there is $n>0$ such that $x\in G^{U,V}_n$ and hence $G^{U,V}=\bigcup_{n=1}^{\infty}G^{U,V}_n$. Consequently, for every $x\in G^{U,V}$ and $y\in H^{U,V}$ there is $n>0$ such that $f^{U,V}_n(x)=0$ and $f^{U,V}_n(y)=1$, that is, $f^{U,V}_n$ \emph{separates} $x$ and $y$.

Put
\[\mathcal{D}=\big\{f^{U,V}_n\colon\ U,V\in\mathcal{B},\ U\neq\emptyset\neq V,\ \overline{U}\cap\overline{V}=\emptyset,\ n>0\big\}.\]
By the argument presented in the previous paragraph, $\mathcal{D}$ separates points, as for each $x\neq y\in M$ we can find $U,V\in\mathcal{B}$ such that $\varphi(x)\in U$, $\varphi(y)\in V$, and $\overline{U}\cap\overline{V}=\emptyset$. Hence, by the Stone--Weierstrass theorem for the compact-open topology (see \cite[Problem 44B]{Willard}), the algebra generated by $\mathcal{D}$ and the constant function $\textbf{1}\colon M\rightarrow\{1\}$ is dense in $C_k(M)$. Hence, the set $\mathcal{E}_M$ of all finite polynomial combinations, with rational coefficients, of members of $\mathcal{D}\cup\{\textbf{1}\}$ is dense in $C_k(M)$. Since $\mathcal{E}_M\subseteq\LIP(M)$, $\mathcal{E}_M$ is also dense in $(\LIP(M),\tau_k)$. As $|\mathcal{E}_M|\le|\mathcal{B}|=ww(M)$, we get that $d(\LIP(M),\tau_k)\le ww(M)$.

\medskip

Since $\mathcal{E}_M$ is dense in $(C(M),\tau_k)$ and $\mathcal{E}_M\subseteq\LIP(M)$, $\LIP(M)$ is dense in $(C(M),\tau_k)$. Consequently, $\LIP(M)$ is dense in $(C(M),\tau)$ for any topology $\tau$ such that $\tau_p\subseteq\tau\subseteq\tau_k$. Thus, (1) holds.

\medskip

Since $\LIP(M)$ is dense in $C_p(M)$, $ww(M)=d(C_p(M))\le d(\LIP(M),\tau_p)$, and hence for any topology $\tau$ on $\LIP(M)$ such that $\tau_p\subseteq\tau\subseteq\tau_k$ we have
\[ww(M)\le d(\LIP(M),\tau_p)\le d(\LIP(M),\tau)\le d(\LIP(M),\tau_k)\le ww(M).\]
Thus, $ww(M)=d(\LIP(M),\tau)$. In particular, we get that $d(M)\ge ww(M)=d(C(M),\tau)=d(\LIP(M),\tau)$, so (2) holds, too.
\end{proof}

\begin{remark}
The inequality in Theorem~\ref{theorem:l_dens_tp_tk}.(2) might be strict. E.g., if $M$ is a discrete metric space of cardinality $2^{\aleph_0}$, then, by the Hewitt--Marczewski--Pondiczery theorem, $d(C_p(M))=d(\mathbb{R}^M)=\aleph_0$, whereas $d(M)=|M|=2^{\aleph_0}>\aleph_0$, and so $d(\LIP(M),\tau)<d(M)$ for every topology $\tau_p\subseteq\tau\subseteq\tau_k$.
\end{remark}

For a metric space $M$, the space $\mathcal{E}_M$---constructed in the proof of Theorem~\ref{theorem:l_dens_tp_tk}---consists of bounded Lipschitz functions on $M$, therefore $\mathcal{E}_M$ is a dense subset of $(\lip(M),\tau_k)$ and hence $d(\lip(M),\tau_k)\le ww(M)$. Also, it follows that $\lip(M)$ is dense in $C_k(M)$ and so in $C_p(M)$, hence $d(\lip(M),\tau_p)\ge d(C_p(M))=ww(M)$. Consequently, we get the following result.

\begin{corollary}\label{cor:lip_dens_tp_tk}
Let $M$ be a metric space. Let $\tau$ be any topology on the space $C(M)$ laying between the pointwise topology and the compact-open topology, i.e., $\tau_p\subseteq\tau\subseteq\tau_k$. Then,
\begin{enumerate}
	\item the space $\lip(M)$ is dense in $(C(M),\tau)$, and
	\item $d(C(M),\tau)=d(\lip(M),\tau)$.
\end{enumerate}
\end{corollary}

For a topological space $X$ and a point $e\in X$, set $C_e(X)=\{f\in C(X)\colon f(e)=0\}$. 

\begin{corollary}\label{cor:lip0_dens_tp_tk}
Let $M$ be a metric space with the base point $e$. Let $\tau$ be any topology on the space $C(M)$ laying between the pointwise topology and the compact-open topology, i.e., $\tau_p\subseteq\tau\subseteq\tau_k$. Then,
\begin{enumerate}
	\item $(\lip_0(M),\tau)$ is dense in $(C_e(M),\tau)$, and
	\item $d(C(M),\tau)=d(C_e(M),\tau)=d(\lip_0(M),\tau)$.
\end{enumerate}
\end{corollary}
\begin{proof}
We work with the topology $\tau$ on $C(M)$. The algebraic direct sum $C(M)=\mathbb{R}\oplus C_e(M)$ is topological (by \cite[Corollary on p. 96]{RobRob}), therefore the canonical projection $P\colon C(M)\rightarrow C_e(M)$ is continuous. Consequently, $P$ transfers dense subsets of $C(M)$ onto dense subsets of $C_e(M)$. Since $P[\LIP(M)]=\lip_0(M)$ and $\LIP(M)$ is dense in $C(M)$ (by Theorem~\ref{theorem:l_dens_tp_tk}.(1)), $\lip_0(M)$ is dense in $C_e(M)$. Thus, (1) holds.

By the continuity of $P$, we also have $d(\lip_0(M),\tau)\le d(\LIP(M),\tau)$. The subset $\mathbb{Q}+\lip_0(M)$ is dense in $(\LIP(M),\tau_p)$, so, by Theorem \ref{theorem:l_dens_tp_tk}.(2), it holds
\[d(\lip_0(M),\tau)\ge d(\lip_0(M),\tau_p)\ge d(\LIP(M),\tau_p)=d(\LIP(M),\tau).\] 
Consequently, $d(\lip_0(M),\tau)=d(\LIP(M),\tau)$. A similar argument shows that $d(C(M),\tau)=d(C_e(M),\tau)$. Hence, (2) holds as well.
\end{proof}

\begin{remark}\label{rem_cp_dens}
In the case of the pointwise topology $\tau_p$ we can provide a direct proof that $d(\lip_0(M)_p)\le d(M)$. Recall that for every infinite-dimensional Banach space $E$ its density with respect to the norm topology is equal to the weight of the dual unit ball $B_{E^*}$ with respect to the weak* topology, i.e., $d(E,\|\cdot\|)=w(B_{E^*},w^*)$. It follows that the unit ball $B_{\lip_0(M)}$ of $\lip_0(M)$ has density at most $d(\mathcal{F}(M))=d(M)$ with respect to the weak* topology and hence with respect to the pointwise topology (as both coincide on bounded subsets of $\lip_0(M)$). It follows that $d(\lip_0(M)_p)\le d(M)$.
\end{remark}

Theorem~\ref{theorem:l_dens_tp_tk} and Corollaries~\ref{cor:lip_dens_tp_tk} and~\ref{cor:lip0_dens_tp_tk} immediately yield the following result.

\begin{theorem}\label{cor:lip_lip0_dens_tp_tk}
Let $M$ be a metric space with the base point $e$. Let $\tau$ be any topology on the space $C(M)$ laying between the pointwise topology and the compact-open topology, i.e., $\tau_p\subseteq\tau\subseteq\tau_k$. Then,
\[ww(M)=d(C(M),\tau)=d(\LIP(M),\tau)=d(\lip(M),\tau)=d(\lip_0(M),\tau)\le d(M).\]
\end{theorem}

\begin{corollary}
Let $M$ be a metric space with the base point $e$. Let
\[(X,Y)\in\big\{(\LIP(M),C(M)),\ (\lip(M),C(M)),\ (\lip_0(M),C_e(M))\big\}.\]
Let $\tau$ be any topology on $Y$ contained between the pointwise topology and the compact-open topology. Then, $X$ is closed in $Y$ with respect to $\tau$ if and only if $M$ is finite.
\end{corollary}
\begin{proof}
If $M$ is finite, then $X=Y$. The other direction follows from the density of $X$ in $Y$ (see the results above) and the existence of non-Lipschitz continuous functions on infinite metric spaces.
\end{proof}

We finish this section with the following folklore remark concerning densities and homeomorphisms of Lipschitz-free spaces.

\begin{remark}\label{rem_fm_homeo}
Recall that for every infinite metric space $M$ we have $d(\mathcal{F}(M))=d(M)$. 
Hence, by the Toru\'nczyk theorem, for two infinite metric spaces $M$ and $N$ the Banach spaces $\mathcal{F}(M)$ and $\mathcal{F}(N)$ are homeomorphic if and only if $d(M)=d(N)$.
\end{remark}

\section{Continuous mappings from and onto spaces related to Lipschitz functions\label{sec:operators}}

Recall that, by a theorem of Grothendieck \cite{Groth52}, for a compact space $K$ every uniformly bounded set in the Banach space $C(K)$  which is compact in $C_p(K)$ is weakly compact in $C(K)$. For any infinite metric space $M$ and its spaces $\lip_0(M)$ and $\lip_0(M)_p$ this theorem fails, as the unit ball $B_{\lip_0(M)}$ of $\lip_0(M)$ is compact in the pointwise topology (see \cite[Theorem 2.37]{Weaver}) but not weakly compact (as $\lip_0(M)$ is not reflexive). We also have the following negative result.

\begin{proposition}\label{prop:not_weakly_lindelof}
Let $M$ be an infinite metric space. Then, both the space $\lip_0(M)$ and its unit ball $B_{\lip_0(M)}$ are not weakly Lindel\"of.
\end{proposition}
\begin{proof}
The space $\lip_0(M)$ is not weakly Lindel\"of, since it contains a closed copy of the space $\ell_\infty$, which is well-known to be not weakly Lindel\"of (see \cite[Example 1.(i)]{Corson}). Further, if $B_{\lip_0(M)}$ was weakly Lindel\"of, then $\lip_0(M)_w$ would be covered by a countable collection of Lindel\"of spaces, hence it would be itself a Lindel\"of space.
\end{proof}

We now proceed with a series of results concerning continuous (not necessarily linear) functions from or onto spaces of Lipschitz functions endowed with various topologies. We start with the following immediate corollary of Theorem~\ref{cor1}. Recall that a topological space $X$ is \emph{extremally disconnected} if the closure of any open subset of $X$ is open.

\begin{corollary}\label{cor:lip0m_ck}
For any infinite metric space $M$ there exists an extremally disconnected compact space $K$ with $d(K)=d(M)$ and $w(K)=2^{d(M)}$ and a continuous linear surjection $T\colon\lip_0(M)\rightarrow C(K)$.
\end{corollary}
\begin{proof}
By Theorem~\ref{cor1} $\lip_0(M)$ contains a complemented isomorphic copy $Y$ of the Banach space $\ell_\infty(d(M))$. Let $I\colon Y\rightarrow\ell_\infty(d(M))$ be an isomorphism and $P\colon\lip_0(M)\rightarrow Y$ a projection. It is well-known that there exists an isomorphism $J\colon\ell_\infty(d(M))\rightarrow C(K)$ of Banach spaces, where $K=\beta\Gamma$ is the \v{C}ech--Stone compactification of a discrete space $\Gamma$ of cardinality $|\Gamma|=d(M)$. $K$ is an extremally disconnected space with $d(K)=d(M)$ and $w(K)=2^{d(M)}$. Set $T=J\circ I\circ P$; then $T\colon\lip_0(M)\rightarrow C(K)$ is a continuous linear surjection.
\end{proof}

Note that for any vector space $E$ and any two linear topologies $\tau\subseteq\tau'$ on $E$ the identity map $(E,\tau')\rightarrow(E,\tau)$ is obviously a continuous linear surjection. In particular, Corollary~\ref{cor:lip0m_ck} yields a continuous linear surjection $\lip_0(M)\rightarrow(C(K),\tau)$ for any topology $\tau$ on $C(K)$ contained between the pointwise topology and the norm topology.

\begin{theorem}\label{theorem_lipp_onto_lipw}
  Let $M$ and $N$  be  infinite metric spaces and let $X$ be an infinite Tychonoff space. Let $\tau$ be any topology on $\lip_0(N)$ containing the weak topology, i.e. $w\subseteq\tau$, and let $\tau'$ be any topology on $C(X)$ containing the pointwise topology, i.e. $\tau_p\subseteq\tau'$. Then, there is no continuous surjection $T\colon\lip_0(M)_p\rightarrow Y$ for the following spaces $Y$:
\begin{enumerate}
	\item $Y=(\lip_{0}(N),\tau)$,
	\item $Y=(C(X),\tau')$.
\end{enumerate}
\end{theorem}
\begin{proof}
By the preceding remark concerning the identity mappings, we can assume that $\tau$ is the weak topology on $\lip_0(N)$ and $\tau'$ is the pointwise topology on $C(X)$, i.e. $\tau=w$ and $\tau'=\tau_p$.

Recall that the closed unit ball $B=B_{\lip_0(M)}$ of $\lip_{0}(M)$ is compact in the pointwise topology. Then, the sets $nB$, $n\in\mathbb{N}$, are also pointwise compact and they cover the space $\lip_{0}(M)_p$. Thus, if there exists a continuous surjection $T\colon \lip_{0}(M)_p\rightarrow \lip_{0}(N)_w$, then $\lip_{0}(N)_w=\bigcup_{n=0}^\infty T[nB]$, where each $T[nB]$ is weakly compact, and hence $\lip_{0}(N)$ is  Lindel\"of in the weak topology, which, by Proposition~\ref{prop:not_weakly_lindelof}, is not true. Thus, (1) holds.

By a similar argument as in (1), if there is a continuous surjection $\lip_{0}(M)_p\rightarrow C_p(X)$, then $C_p(X)$ can be covered by a sequence of compact sets, which, by \cite[Theorem I.2.1]{Arch}, means that $X$ is finite, which is a contradiction. Thus, (2) holds, too.
\end{proof}

For spaces $\lip_{0}(M)_p$ with $M$ separable we have the following sequential result, being a formal strengthening of Theorem~\ref{theorem_lipp_onto_lipw}.(1).

\begin{proposition}\label{prop_lipp_seq_onto_lipw}
Let $M$ and $N$ be infinite metric spaces. If $M$ is separable, then there is no sequentially continuous surjection $T\colon\lip_{0}(M)_p\rightarrow\lip_{0}(N)_w$.
\end{proposition}
\begin{proof}
Suppose that $M$ is separable and there is a sequentially continuous surjection $T\colon \lip_{0}(M)_p \rightarrow \lip_{0}(N)_w$. The closed unit ball of $\lip_{0}(M)$ is compact and, by the separability of $M$ and hence of $\mathcal{F}(M)$, metrizable with respect to the pointwise topology, which implies that $\lip_{0}(M)_p$ is covered by a sequence of sequentially compact subsets. Consequently, $\lip_{0}(N)_w$  is covered by a sequence of sequentially compact subsets, say $(B_{n})_{n=0}^\infty$. By the Eberlein--\v{S}mulian theorem, each set $B_{n}$ is compact in  $\lip_{0}(N)_w$. But this implies that $\lip_{0}(N)_w$ is Lindel\"of, a contradiction with Proposition~\ref{prop:not_weakly_lindelof}.
\end{proof}

If $K$ is an infinite compact space and $N$ is an infinite metric space such that $w(K)<2^{d(N)}$, then there is no continuous surjection $T\colon C(K)\rightarrow\lip_0(N)$ (between Banach spaces), as $d(C(K))=w(K)<2^{d(N)}=d(\lip_0(N))$. The next theorem generalizes this observation.

\begin{theorem}\label{theorem:cpm_lip0n}
  Let $X$ be an infinite Tychonoff space and $N$ an infinite metric space such that $w(X)<2^{d(N)}$. Let $\tau$ be any topology on $C(X)$ contained between the pointwise topology and the compact-open topology, i.e. $\tau_p\subseteq\tau\subseteq\tau_k$, and let $\tau'$ be any topology on $\lip_0(N)$ containing the weak topology, i.e. $w\subseteq\tau'$. Then, there is no continuous surjection $T\colon(C(X),\tau)\rightarrow (\lip_0(N),\tau')$.
\end{theorem}
\begin{proof}
We showed in Corollary~\ref{To} that $d(\lip_0(N)_w)=d(\lip_{0}(N))=2^{d(N)}$. Consequently we have, $d(\lip_0(N),\tau')\ge2^{d(N)}$. On the other hand, we also have  $d(C(X),\tau)=ww(X)\le w(X)$ (by \cite[Theorem 1]{Nob71}, see the previous section), so $d(\lip_0(N),\tau')>d(C(X),\tau)$, hence there cannot exist any continuous mapping from $(C(X),\tau)$ onto $(\lip_0(N),\tau')$.
\end{proof}

Recall that $\lip_0([0,1])$ is isomorphic to $\ell_\infty$, so in particular to the Banach space $(C(\beta\mathbb{N}),\|\cdot\|_\infty)$. Since $w(\beta\mathbb{N})=2^{\aleph_0}=2^{d([0,1])}$, it follows that in Theorem~\ref{theorem:cpm_lip0n} the assumption that $w(X)<2^{d(N)}$ cannot be relaxed to $w(X)\le2^{d(N)}$. On the other hand, C\'{u}th, Doucha, and Wojtaszczyk \cite[proof of Theorem 4.1]{Doucha} showed that there is no continuous linear surjection $T\colon C(K)\rightarrow\lip_0([0,1]^2)$ for any compact space $K$. Since the space $[0,1]^2$ is an absolute Lipschitz retract (see Section~\ref{sec:ell_1} for the definition), the latter result implies that for any metric space $M$ containing $[0,1]^2$ there is no continuous linear surjection $T\colon C(K)\rightarrow\lip_0(M)$ for any compact space $K$ (as $\lip_0([0,1]^2)$ is complemented in $\lip_0(M)$, see Proposition~\ref{prop:Retract1}).

\subsection{Continuous mappings from Lipschitz-free spaces\label{sec:lipschitz_free}}

In this section we will briefly study the existence of continuous mappings from Lipschitz-free spaces $\mathcal{F}(M)$ (in particular, with respect to the weak topology). We start with the following proposition.

\begin{proposition}\label{prop:fm_lip0n}
Let $M$ be an infinite metric space. Let $E$ be a Banach space and let $\tau$ be any topology on $E$ contained between the weak topology and the norm topology. The following conditions are equivalent:
\begin{enumerate}
	\item there is a continuous linear surjection $T\colon\mathcal{F}(M)\rightarrow E$,
	\item there is a continuous surjection $T\colon \mathcal{F}(M)\rightarrow E$,
	\item there is a continuous surjection $T\colon \mathcal{F}(M)\rightarrow (E,\tau)$,
	\item $d(M)\ge d(E)$,
	\item $\mathcal{F}(M)$ contains a complemented subspace homeomorphic to $E$,
	\item $\mathcal{F}(M)$ contains a subset homeomorphic to $E$.
\end{enumerate}
\end{proposition}
\begin{proof}
(1)$\Rightarrow$(2) Obvious.

(2)$\Rightarrow$(3) Consider the identity mapping $(E,\|\cdot\|)\rightarrow(E,\tau)$.

(3)$\Rightarrow$(4) Continuous surjections do not increase densities, so we have $d(M)=d(\mathcal{F}(M))\ge d(E,\tau)\ge d(E,w)=d(E)$.

(4)$\Rightarrow$(5) By Corollary~\ref{hano_ell1}, the space $\mathcal{F}(M)$ contains a complemented copy of $\ell_1(d(M))$, hence if $d(M)\ge d(E)$, then $\ell_1(d(E))$ is also isomorphic to a  complemented subspace $Y$ of $\mathcal{F}(M)$. Since $d(Y)=d(\ell_1(d(E)))=d(E)$, the Toru\'nczyk theorem implies that $Y$ is homeomorphic to $E$.

(5)$\Rightarrow$(6) Obvious.

(6)$\Rightarrow$(4) Let $Y$ be a subset of $\mathcal{F}(M)$ homeomorphic to $E$. Since $\mathcal{F}(M)$ is a metric space, $d(E)=d(Y)\le d(\mathcal{F}(M))=d(M)$.

(4)$\Rightarrow$(1) As in (4)$\Rightarrow$(5), note that $\mathcal{F}(M)$ contains a complemented subspace isomorphic to $\ell_1(d(E))$. Since $d(\ell_1(d(E)))=d(E)$, there is a continuous linear surjection $S\colon\ell_1(d(E))\rightarrow E$, so in particular there is a linear continuous surjection $T\colon\mathcal{F}(M)\to E$.
\end{proof}

We will need the following auxiliary notions. A Tychonoff space $X$ is called \emph{Ascoli} if every compact subset $\mathcal{K}$ of $C_k(X)$ is evenly continuous, that is, the map $X\times\mathcal{K}\ni(x,f) \mapsto f(x)\in\mathbb{R}$ is continuous (see \cite{BGAscoli} and \cite{SGKZ} for details). A family $\mathcal{D}$ of subsets of a topological space $X$ is called a \emph{$k$-network in $X$} if for every subsets $K\subseteq U\subseteq X$ with $K$ compact and $U$ open in $X$ we have $K\subseteq \bigcup\mathcal{F}\subseteq U$ for some finite family $\mathcal{F}\subseteq\mathcal{D}$. A topological space $X$ is said to be an \emph{$\aleph_0$-space} if $X$ is regular and has a countable $k$-network. By~\cite{Mich}, every metrizable separable Tychonoff space is an $\aleph_0$-space.

The next proposition provides a range of examples of metric spaces whose Lipschitz-free Banach spaces are not homeomorphic with respect to the weak topologies.

\begin{proposition}\label{ex}
  Let $M=C(K)$ for an uncountable compact metric space $K$. Let $N$ be a separable infinite ultrametric space (e.g. $N$ is the Cantor space $2^\mathbb{N}$). Then the Lipschitz-free Banach spaces $\mathcal{F}(M)$ and $\mathcal{F}(N)$ are homeomorphic but the spaces $\mathcal{F}(M)_w$ and $\mathcal{F}(N)_w$ are not homeomorphic.
\end{proposition}
\begin{proof}
Both spaces $M$ and $N$ have the same density (that is, both are separable), so by Remark~\ref{rem_fm_homeo} and the Toru\'nczyk theorem the spaces $\mathcal{F}(M)$ and $\mathcal{F}(N)$ are homeomorphic. 

Since $M$ is a separable Banach space, \cite[Theorems~2.12 and~3.1]{GK} imply that $M$ embeds linearly into $\mathcal{F}(M)$ (as a closed subspace), so  $\mathcal{F}(M)_w$ contains a linear copy of $C(K)_{w}$. Assume that $\mathcal{F}(M)_w$ and $\mathcal{F}(N)_w$ are homeomorphic. Recall that $\mathcal{F}(N)$ is isomorphic to the space $\ell_{1}$ (by
\cite[Theorem~2]{C-D}). By \cite[Theorem 1.2]{GKKM} the space $(\ell_{1})_{w}$ is an $\aleph_{0}$-space. Consequently, $\mathcal{F}(M)_w$ is an $\aleph_{0}$-space, too.  Since the property of being an $\aleph_{0}$-space is inherited by subspaces, the space $C(K)_w$ is also an $\aleph_{0}$-space, which is a contradiction to~\cite[Proposition 10.8]{Mich}.
\end{proof}

The next theorem is motivated by recent research related to the open question asking whether there exist two infinite compact spaces $K$ and $L$ such that $C(K)_w$ and $C_p(L)$ are homeomorphic; see \cite{Krupski-1} and \cite{K-M} for details and references. Recall that a Banach space $E$ has the \emph{Schur property} if every weakly convergent sequence in $E$ is norm convergent. For several characterizations of Lipschitz-free Banach spaces with the Schur property, see \cite{AGPP}.

\begin{theorem}\label{theorem_cpm_homeo_fmw}
Let $M$ and $N$ be infinite metric spaces such that at least one of the following conditions holds:
\begin{enumerate}
	\item at least one of the spaces $M$ and $N$ is not separable, or
	\item $M$ is countable, or
	\item $M$ is scattered, or
	\item $M$ is not $\sigma$-compact, or
	\item $\mathcal{F}(N)$ has the Schur property, or
	\item $\mathcal{F}(N)$ isomorphically embeds into some separable space $L_1(\Omega,\Sigma,\mu)$.
\end{enumerate}Then, the spaces $C_p(M)$ and $\mathcal{F}(N)_w$ are not homeomorphic.
\end{theorem}
\begin{proof}
Suppose that (1) holds. Assume first that $d(M)>\aleph_0$ and $d(N)>\aleph_0$. Then, $M$ is not Lindel\"{o}f and hence by the Pytkeev theorem the tightness of $C_p(M)$ is uncountable (see \cite[Exercise 149]{Tka1}). On the other hand, by the Kaplansky theorem (see \cite[Theorem 3.54]{Czech10}), $\mathcal{F}(M)_w$ has countable tightness.

If $d(M)=\aleph_0<d(N)$, then $d(C_p(M))=\aleph_0<d(\mathcal{F}(N))=d(\mathcal{F}(N)_w)$.

Finally, assume that $d(M)>\aleph_0=d(N)$. Since $M$ is not separable, it is not second countable and hence $C_p(M)$ is not Lindel\"{o}f (see \cite[Exercise 215]{Tka1}). Since $d(\mathcal{F}(N))=d(N)=\aleph_0$, $\mathcal{F}(N)$ is separable and hence $\mathcal{F}(N)_w$ is Lindel\"{o}f.

It follows that in all of the above three cases $C_p(M)$ and $\mathcal{F}(M)_w$ are not homeomorphic.

\medskip

If $M$ is countable, so (2) holds, then $C_p(M)$ is metrizable (being a subset of $\mathbb{R}^M$). Since no infinite-dimensional Banach space is metrizable in the weak topology, $C_p(M)$ and $\mathcal{F}(N)_w$ cannot be homeomorphic.

\medskip

If (3) holds, that is, $M$ is scattered, then by \cite[Theorem 1.3]{SGKZ} the space $C_p(M)$ is Ascoli. Suppose that $C_p(M)$ and $\mathcal{F}(N)_w$ are homeomorphic. Since the Ascoli property is inherited by continuous open maps (see \cite[Proposition 3.3]{SKP}), the space $\mathcal{F}(N)_w$ is also an Ascoli space. On the other hand, if the weak topology of a normed space $E$ is Ascoli, then $E$ is finite-dimensional (see \cite[Proposition 3.2]{SKP}), which is impossible in the case of $\mathcal{F}(N)$ for infinite $N$.

\medskip

Assume that $M$ is not $\sigma$-compact, that is, (4) holds. By (1) we can assume that $\mathcal{F}(N)$ is separable, so, in particular, that it is a Polish space. Consequently, $\mathcal{F}(N)$ is analytic and hence $\mathcal{F}(N)_w$ is also analytic (as the identity is norm--weakly continuous). Were $C_p(M)$ and $\mathcal{F}(N)_w$ homeomorphic, $C_p(M)$ would be an analytic space as well. But, by \cite[Theorem~2.1.3]{Calbrix}, it would mean that $M$ is $\sigma$-compact, a contradiction.

\medskip

Finally, assume that (5) or (6) holds. By (1) and (2) we only need to check the case when $|M|>\aleph_0$ and $d(\mathcal{F}(N))=\aleph_0$. But then, by \cite[Proposition 4.4]{SW91}, $\mathcal{F}(N)_w$ is an $\aleph_0$-space, whereas, 
since $M$ is uncountable, $C_p(M)$ is not an $\aleph_0$-space, by \cite[Proposition 10.7]{Mich}. It follows that $C_p(M)$ and $\mathcal{F}(N)_w$ cannot be homeomorphic.
\end{proof}

We deal with Lipschitz-free Banach spaces having the Schur property in the next subsection. Note that, by \cite[Theorem C]{AGPP}, for no non-trivial Banach space $N$ the space $\mathcal{F}(N)$ can have the Schur property as in this case $N$ contains bilipschitz copies of subsets of the real line $\mathbb{R}$ of non-zero Lebesgue measure (e.g. of $\mathbb{R}$ itself). 

\subsection{The Schur property and sequential homeomorphisms\label{sec:schur}}

In this subsection we will study the Schur property of Lipschitz-free Banach spaces $\mathcal{F}(M)$ in the context of the existence of sequentially continuous mappings. Note that the Schur property is inherited by closed subspaces, therefore no Banach space of the form $\lip_0(M)$ for infinite $M$ has it (as $\lip_0(M)$ contains a copy of $\ell_\infty$).

Recall that two topological spaces $X$ and $Y$  are \emph{sequentially homeomorphic} if there exists a sequentially continuous bijection $f\colon X\rightarrow Y$ such that the inverse $f^{-1}$ is also sequentially continuous; such $f$ is then said to be a \emph{sequential homeomorphism}. It is well-known that no infinite-dimensional Banach space $E$ is homeomorphic to $E_w$ as the latter is non-metrizable. It may however happen that $E$ and $E_w$ are sequentially homeomorphic and this happens precisely when $E$ has the Schur property.

\begin{theorem}\label{lee}
Let $E$ be a Banach space. Then, $E$ has the Schur property if and only if $E$ and $E_w$ are sequentially homeomorphic.
\end{theorem}
\begin{proof}
Let $T\colon E\rightarrow E_w$ be a sequential homeomorphism and assume that the Banach space  $E$  fails to have the Schur property. Recall that we write $w=\sigma(E,E^*)$ and let $\tau$ denote the original norm topology on $E$.

Let $\gamma$ be the family of all weakly sequentially open subsets of $E$, that is, for a subset $U$ of $E$ we have: $U\in\gamma$ if and only if, for each sequence $(x_n)_{n=0}^\infty$ in $E$ weakly convergent to some $x\in U$, for almost all $n\in\mathbb{N}$ we have $x_{n}\in U$. Then, it is easy to see that $\gamma$ is a topology on $E$ such that $w\subseteq\gamma\subseteq\tau$. Of course, as also easily seen, both topologies $w$ and $\gamma$ have the same convergent sequences. Therefore, $T\colon(E,\tau)\rightarrow (E,\gamma)$ is also a homeomorphism. Indeed, if $U\subseteq E$ is $\tau$-open, then $U$ is $\tau$-sequentially open and so $T[U]$ is weakly sequentially open, so $\gamma$-open. Conversely, if $T[U]$ is $\gamma$-open, then it is weakly sequentially open. Thus, $U$ is $\tau$-sequentially open and hence $\tau$-open, as $\tau$ is metrizable.

It follows that $(E,\gamma)$ is a Baire space. Let $B_{E}$ be the closed unit ball of (the Banach space) $E$. Then, for every $m\in\mathbb{N}$, $mB_E$ is closed in $w$ and so in $\gamma$. By the Baire property of $\gamma$, there exist $m\in\mathbb{N}$, $m>0$, a point $y\in mB_E$, and an open neighbourhood $U_y$ of $y$ in the topology  $\gamma$ such that $U_y\subseteq mB_E$. Note that $\|y\|<m$.

Since $E$ does not have the Schur property, there exist $\varepsilon>0$ and a sequence $(y_{n})_{n=0}^\infty$ such that $y_{n}\rightarrow 0$ weakly but $\|y_{n}\|\ge\varepsilon$ for each $n\in\mathbb{N}$. For every $n\in\mathbb{N}$ let
\[z_n=\frac{2m}{\|y_n\|}y_n,\]
so $\|z_n\|=2m$ and
\[\|z_n+y\|\ge\|z_n\|-\|y\|>2m-m=m,\]
hence $z_n+y\not\in mB_E$. On the other hand, for every $z^*\in E^*$ we have
\[|z^*(z_n)|=\frac{2m}{\|y_n\|}|z^*(y_n)|\le\frac{2m}{\varepsilon}|z^*(y_n)|\to0,\]
so $z_n\to0$ weakly.

It follows that $z_n+y\to y$ weakly and hence $z_{n}+y\rightarrow y$ in $\gamma$, as both topologies $w$ and $\gamma$ have the same convergent sequences. This provides a contradiction since for almost all $n\in\mathbb{N}$ we have  $z_{n}+y\in U_y\subseteq mB_E$, whereas $z_{n}+y\not\in mB_E$ for all $n\in\mathbb{N}$.

The other implication is trivial as in this case the identity mapping is a sequential homeomorphism.
\end{proof}

It is well-known that the space $\ell_1$ has the Schur property whereas the space $L_1([0,1])$ does not. Thus, for these two Banach spaces, important from the point of view of the study of the Lipschitz-free spaces, we have the following corollary.

\begin{corollary}
The spaces $\ell_1$ and $(\ell_1)_w$ are sequentially homeomorphic whereas the spaces $L_1([0,1])$ and $L_1([0,1])_w$ are not.
\end{corollary}

For the class of Banach spaces with the Schur property, we obtain the following characterization of the existence of sequential homeomorphisms. The theorem extends the results presented in \cite{Banakh}, which only apply to Banach spaces with separable dual spaces---note that for any infinite-dimensional Banach space $E$ with the Schur property its dual $E^*$ is never separable as $E$ is $\ell_1$-saturated and therefore $E^*$ admits a continuous linear surjection onto the space $\ell_\infty$.

\begin{lemma}\label{lem:seqhom}
The following hold for any Banach spaces $E$ and $F$:
\begin{enumerate}
	\item $d(E)=d_{seq}(E_w)$,
	\item if $E_w$ and $F_w$ are sequentially homeomorphic, then $d(E)=d(F)$.
\end{enumerate}
\end{lemma}
\begin{proof}
(1) Any weakly sequentially dense set is weakly dense, so $d_{seq}(E_w)\ge d(E_w)=d(E)$. On the other hand, since the weak topology is weaker than the norm topology, any norm sequentially dense set is weakly sequentially dense, so $d_{seq}(E)\ge d_{seq}(E_w)$. But since for metric spaces dense sets are sequentially dense, we get $d(E)\ge d_{seq}(E)\ge d_{seq}(E_w)$.

(2) Since sequential homeomorphisms preserve sequential density, by (1) we have
\[d(E)=d_{seq}(E_w)=d_{seq}(F_w)=d(F).\]
\end{proof}

\begin{theorem}\label{theorem:weak_seq_homeo}
Let $E$ and $F$ be Banach spaces, both having the Schur property. Then, the spaces $E_w$ and $F_w$ are sequentially homeomorphic if and only if $d(E)=d(F)$.
\end{theorem}
\begin{proof}
Assume that $d(E)=d(F)$. By the Toru\'{n}czyk theorem, there exists a homeomorphism $S\colon E\rightarrow F$ of the Banach spaces. Of course, $S$ is also a sequential homeomorphism. Since both $E$ and $F$ have the Schur property, the identity mappings $I_1\colon E_w\rightarrow E$ and $I_2\colon F\rightarrow F_w$ are sequential homeomorphisms. Taking the composition $T=I_2 S I_1$, we get a sequential homeomorphism $E_w\rightarrow F_w$.

The converse implication follows from Lemma \ref{lem:seqhom}.(2) (and does not require the Schur property).
\end{proof}

The following proposition provides a wide class of examples of pairs of Banach spaces $E$ and $F$ such that the spaces $E_w$ and $F_w$ are  not homeomorphic.

\begin{proposition}
  Let $E$ and $F$ be separable Banach spaces. Assume that $E^*$ is separable and $F^*$ is non-separable.
  Then, $E_w$ and $F_w$ are not homeomorphic.
\end{proposition}
\begin{proof}
Since $E^*$ is separable, the closed unit ball $B_E$ of $E$ is metrizable in the weak topology and hence, for every $n\in\mathbb{N}$, the set $B_n=n B_E$ is metrizable in the weak topology. On the other hand, the closed unit ball $B_F$ of $F$ is not metrizable in the weak topology. Therefore, the spaces $E_w$ and $F_w$ are not  homeomorphic. To see this, assume there is a homeomorphism $T\colon E_w\rightarrow F_w$. Then, for each $n\in\mathbb{N}$ the set $T[B_n]$ is a closed metrizable subset of $F_w$. In particular, each $T[B_n]$ is closed in $F$ (in the norm topology), so since $F=\bigcup_{n=0}^\infty T[B_n]$, by the Baire Category Theorem, there is $n_0\in\mathbb{N}$ such that $T[B_{n_0}]$ contains a closed ball $B$ of $F$. Since $T[B_{n_0}]$ is metrizable in the weak topology, $B$ is also metrizable in the weak topology, and hence $B_F$ is as well, a contradiction.
\end{proof}

Taking into account that by \cite[Corollary~2.7]{Aliaga} the Lipschitz-free space on a complete discrete metric space has the Schur property, Theorem~\ref{lee} implies the following characterization of the Schur property for Lipschitz-free spaces.

\begin{theorem}\label{theorem_schur_discrete}
Let $M$ be an infinite metric space. The following are equivalent:
\begin{enumerate}
	\item the space $\mathcal{F}(M)$ has the Schur property,
	\item for every complete discrete metric space $N$ with cardinality $d(M)$, the spaces $\mathcal{F}(M)_w$ and $\mathcal{F}(N)_w$ are sequentially homeomorphic,
	\item there is a complete discrete metric space $N$ with cardinality $d(M)$ for which the spaces $\mathcal{F}(M)_w$ and $\mathcal{F}(N)_w$ are sequentially homeomorphic.          
\end{enumerate}
\end{theorem}
\begin{proof}
Assume (1), that is, that $\mathcal{F}(M)$ has the Schur property and let $N$ be a complete discrete metric space of cardinality $d(M)$. By the Toru\'nczyk theorem, the spaces $\mathcal{F}(M)$ and $\mathcal{F}(N)$ are homeomorphic, hence sequentially homeomorphic. Since both $\mathcal{F}(M)$ and $\mathcal{F}(N)$ have the Schur property, by Theorem~\ref{lee}, the spaces $\mathcal{F}(M)$ and $\mathcal{F}(M)_w$ are sequentially homeomorphic and the spaces $\mathcal{F}(N)$ and $\mathcal{F}(N)_w$ are sequentially homeomorphic. Now, (2) follows by composing appropriate sequential homeomorphisms.

Implication (2)$\Rightarrow$(3) is clear.

In order to show that (3) implies (1), let $N$ be a complete discrete metric space for which $\mathcal{F}(M)_w$ and $\mathcal{F}(N)_w$ are sequentially homeomorphic. Since $\mathcal{F}(N)$ has the Schur property, by Theorem~\ref{lee} also the spaces $\mathcal{F}(N)$ and $\mathcal{F}(N)_w$ are sequentially homeomorphic. Since $\mathcal{F}(M)$ and $\mathcal{F}(N)$ have the same density, they are homeomophic by Toruńczyk's theorem. Composing the corresponding mappings yields a sequential homeomorphism between $\mathcal{F}(M)_w$ and $\mathcal{F}(M)$, which, again by Theorem \ref{lee}, yields that $\mathcal{F}(M)$ has the Schur property.
\end{proof}

\begin{remark}
For every infinite cardinal number $\kappa$, let $N_\kappa$ be a metric space such that $d(N_\kappa)=\kappa$ and the Lipschitz-free space $\mathcal{F}(N_\kappa)$ has the Schur property. Then, by its proof, Theorem \ref{theorem_schur_discrete} can be rephrased in the following way: \emph{For every metric space $M$, the space $\mathcal{F}(M)$ has the Schur property if and only if $\mathcal{F}(M)_w$ and $\mathcal{F}(N_{d(M)})_w$ are sequentially homeomorphic.}
\end{remark}

Recall that the Josefson--Nissenzweig theorem asserts that for every infinite-dimensional Banach space $E$ there exists a sequence $(x_{n}^{*})_{n=0}^\infty$ in the dual space $E^*$ which is weak* convergent to $0$ and such that $\|x_{n}^{*}\|=1$ for all $n\in\mathbb{N}$. Applying the argument presented in the proof of Theorem~\ref{lee} and the Josefson--Nissenzweig theorem one gets the following results (where, recall, $w^*$ denotes the weak* topology).

\begin{corollary}\label{co}
Let $E$ be an infinite-dimensional Banach space. Then, the spaces $E^*$ and $(E^*,w^*)$  are not sequentially homeomorphic.
\end{corollary}

\begin{corollary}\label{co2}
Let $M$ be an infinite metric space. Then, the spaces $\lip_0(M)$ and $(\lip_0(M),w^*)$ are not sequentially homeomorphic.
\end{corollary}

\section{Operators from spaces $\lip_0(M)$ onto $\ell_1$\label{sec:ell_1}}

In this section we will study the question for which metric spaces $M$ there exists a continuous linear surjection from the space $\lip_0(M)$ onto the Banach spaces $c_0$ or $\ell_1$. Recall that Rosenthal \cite[Theorem~4.22]{HajekBi} proved that the Banach space $\ell_\infty(\kappa)$, where $\kappa$ is an infinite cardinal number, always admits a continuous linear surjection onto the Hilbert space $\ell_2(2^\kappa)$, so in particular onto the space $\ell_2$. Since for any infinite metric space $M$, by Theorem~\ref{cor1}, $\lip_0(M)$ contains a complemented copy of $\ell_\infty(d(M))$, Rosenthal's result yields the following observation, being a main motivation for this section.

\begin{corollary}\label{separable}
For any infinite metric space $M$ there is a continuous linear surjection $T\colon\lip_0(M)\rightarrow\ell_2(2^{d(M)})$. In particular, there is a continuous linear surjection $S\colon\lip_0(M)\rightarrow\ell_2$ and hence $\lip_0(M)$ has a separable quotient.
\end{corollary}

In what follows, by saying that a Banach space $E$ is a \emph{predual} of a Banach space $F$ we mean that $F$ is isometrically isomorphic to the dual space $E^*$. For a Banach space $E$ and a natural number $n\ge0$ we set: $E^{*(0)}=E$ and $E^{*(n+1)}=(E^{*(n)})^*$.

Note that there is no continuous linear surjection from $\ell_2$ onto $c_0$ or any space $\ell_p$ with $p\neq2$ (since they are not Hilbert spaces). Note also that for any infinite metric space $M$ the Lipschitz-free space $\mathcal{F}(M)$ and the dual space $\lip_0(M)^*$ both contain complemented copies of $\ell_1$ (in the latter case by Theorem \ref{cor1} and  the fact that the dual space $\ell_\infty^*$ contains such a copy). It follows that, for any $n\ge0$, the $n$-th dual space $\lip_0(M)^{*(n)}$ contains a complemented copy of $\ell_1$ if $n$ is odd, and a complemented copy of $\ell_\infty$ if $n$ is even.

Recall that, for any Banach space $E$, the dual space $E^*$ is complemented in the third dual space $E^{***}$ and so in the $(2n+1)$-th dual space $E^{*(2n+1)}$ for any $n\ge0$, and that, by \cite[Theorem 2]{Lindenstrauss} or \cite[Corollary 5.4]{Weaver}, the space $\lip_0(E)$ contains a $1$-complemented copy of $E^*$. 
We thus have the following result.

\begin{proposition}\label{prop:lip0dualell1}
Let $E$ be a Banach space such that the dual space $E^*$ admits a continuous  operator onto $\ell_1$. Then, for any $n\ge 0$, the space $\lip_0(E^{*(2n)})$ admits such an operator, too.

In particular, if $E$ is a predual of $\ell_1$ (e.g. $E=c_0$), then there is a continuous linear surjection from $\lip_0(E)$ onto $\ell_1$.
\end{proposition}

\begin{corollary}
If $E$ is a separable Banach space with no non-trivial cotype, then $\lip_0(E)$ admits a continuous linear surjection onto $\ell_1$.
\end{corollary}
\begin{proof}
  By \cite[Theorem 3.1]{C-C-D}, $\lip_0(E)$ contains a complemented copy of $\lip_0(c_0)$ and thus the claim follows by Proposition~\ref{prop:lip0dualell1}.
\end{proof}

The second statement in the above proposition can be strengthened in the case of $E=c_0$ as follows, using the lifting property of separable Banach spaces.

\begin{lemma}\label{lemma:Lip0CKNotGP}
Let $E$ be a separable Banach space. If $E$ contains an isomorphic copy of a predual $F$ of $\ell_1$, then there is a continuous linear surjection $T\colon\lip_0(E)\rightarrow \ell_1$.

In particular, if $E$ contains an isometric (not necessarily linear) copy of $c_0$, then there is such an operator.
\end{lemma}

\begin{proof}
Since separable Banach spaces have the lifting property, see \cite[Theorem 3.1]{GK}, $E$ and hence also $F$ embeds linearly into $\mathcal{F}(E)$, by \cite[Theorem 2.12]{GK}. By transposition, we obtain a continuous linear surjection $\lip_0(E)\rightarrow \ell_1$.  
The second statement follows from \cite[Corollary 3.3]{GK} saying that a Banach space containing an isometric copy of $c_0$ also contains a linear subspace isometrically isomorphic to $c_0$.
\end{proof}

Recall that a metric space $M$ is an \emph{absolute Lipschitz retract} if $M$ is a Lipschitz retract of every metric space containing $M$. Equivalently, a metric space $M$ is an absolute Lipschitz retract if for all metric spaces $P\subseteq N$ and every Lipschitz mapping $f\colon P\rightarrow M$ there is a Lipschitz extension $F\colon N\rightarrow M$ of $f$. This characterization immediately implies that bilipschitz copies of absolute Lipschitz retracts are absolute Lipschitz retracts. Lindenstrauss \cite{Lindenstrauss} proved that Banach spaces $c_0$ and $C(K)$ for $K$ metric compact are absolute Lipschitz retracts. We refer the reader to \cite[Section 1.1]{BenyaminiLindenstrauss} for an introduction to absolute Lipschitz retracts.

The following fact is folklore.

\begin{proposition}\label{prop:Retract1}
  Let $N$ be a metric space and $M\subseteq N$. If $M$ is a Lipschitz retract of $N$, then the space $\lip_0(M)$ is isomorphic to a complemented subspace of $\lip_0(N)$.

  In particular, if $M$ is an absolute Lipschitz retract such that $\lip_0(M)$ admits a continuous operator onto $\ell_1$, then for every metric space $N$ containing a bilipschitz copy of $M$ the space $\lip_0(N)$ admits such an operator, too.
\end{proposition}

\begin{proof}
  Let $r\colon N\rightarrow M$ be a Lipschitz retraction. Consider the operators
  \[
    J\colon \lip_0(M) \rightarrow \lip_0(N), \quad f \mapsto f \circ r, \qquad\text{and}\qquad T\colon \lip_0(N)\rightarrow\lip_0(M), \quad g\mapsto g|_{M},
  \]
  and observe that for every $f\in\lip_0(M)$ and $g\in\lip_0(N)$ we have
  \[
    \|f\circ r\|_{\lip_0(N)} \le \lip(r)\cdot\|f\|_{\lip_0(M)}, \qquad \|f\circ r\|_{\lip_0(N)} \ge \|f\|_{\lip_0(M)},
  \]
  and $\|g|_M\|_{\lip_0(M)} \le \|g\|_{\lip_0(N)}$. In other words, both the operators are continuous and $J$ is an isomorphism onto its image. Also, by the McShane Extension Theorem, $T$ is onto $\lip_0(M)$. Moreover, $(JT)^2=JT$, and hence $JT$ is a projection onto the image of $J$. This shows that $\lip_0(M)$ is isomorphic to a complemented subspace of $\lip_0(N)$. 
\end{proof}

Having established the above observations and propositions, we are ready to prove the following theorem.

\begin{theorem}\label{thm:BiLipCK}
Let $E$ be a separable Banach space which is an absolute Lipschitz retract and contains an isomorphic copy of $c_0$ (e.g. $E$ is isomorphic to $C(K)$ for metric compact $K$). If a metric space $M$ contains a bilipschitz copy of the unit sphere $S_{E}$ of $E$, then the space $\lip_0(M)$ admits a continuous operator onto $\ell_1$.
\end{theorem}

\begin{proof}
Since  $c_0$ is separably injective, $E$ is isomorphic to $E\oplus\mathbb{R}$, so by the main result of ~\cite{C-K} the spaces $\lip_0(S_E)$ and $\lip_0(E)$ are isomorphic. Since $E$ contains an isomorphic copy of $c_0$, Lemma~\ref{lemma:Lip0CKNotGP} yields that there is a continuous linear surjection from $\lip_0(E)$ onto $\ell_1$ and hence from $\lip_0(S_E)$ onto $\ell_1$.

It is easy to see that the unit ball $B_E$ is a retract of $E$, so consequently, by~\cite{B-S}, the unit sphere $S_E$ is also a Lipschitz retract of $E$. Ultimately, as $E$ is an absolute Lipschitz retract, $S_E$ is an absolute Lipschitz retract, too. Now, the claim follows from Proposition~\ref{prop:Retract1} and the previous paragraph.
\end{proof}

\begin{corollary}\label{cor:bilipschitz_c0}
If a metric space $M$ contains a bilipschitz copy of the unit sphere $S_{c_0}$ of $c_0$, then $\lip_0(M)$ admits a continuous operator onto $\ell_1$.
\end{corollary}

\begin{remark}
Note that in Theorem~\ref{thm:BiLipCK} we may require that $E$ contains an isomorphic copy of a predual of $\ell_1$, as every such predual contains a copy of $c_0$ (by Johnson and Zippin \cite{JZ73}). Note also that Benyamini and Lindenstrauss \cite{BL72} constructed a predual of $\ell_1$ which is not isomorphic to a complemented subspace of any space $C(K)$ (though by \cite[Theorem]{JZ73} it is isomorphic to a quotient of the space $C(2^\mathbb{N})$, where $2^\mathbb{N}$ denotes the Cantor space).
\end{remark}

\begin{corollary}\label{cor:l1inXnoGrothendieck}
  Let $E$ be a Banach space containing a complemented copy of $\ell_1$. Then, $\lip_0(E)$ contains a complemented copy of $\lip_0(\ell_1)$ and so, in particular, of $\ell_1$.
\end{corollary}

\begin{proof}
By the result of Dalet \cite[Proposition 79]{Dalet2}, the space $\lip_0(\ell_1)$ contains a complemented copy of $\ell_1$. Since a continuous linear projection is in particular a Lipschitz retraction, the conclusion follows from Proposition~\ref{prop:Retract1}.
\end{proof}

Using the above results, we now provide a list of several examples of Banach spaces $E$ such that $\lip_0(E)$ has a quotient isomorphic to $\ell_1$. Below by $\mathcal{L}(Y,Z)$ we denote the Banach space of all continuous operators from a normed space $Y$ into a Banach space $Z$, and by $\mbox{Bil}(Y\times W,Z)$ the Banach space of all continuous bilinear mappings from the product $Y\times W$ of Banach spaces $Y$ and $W$ into a Banach space $Z$.

\begin{theorem}\label{theorem:examples_onto_ell1}
Let $n\ge0$. If $E$ is an infinite-dimensional Banach space satisfying any of the following properties, then $\lip_0(E)$ admits a continuous  operator onto $\ell_1$:
\begin{enumerate}[itemsep=1mm]
	\item $E=C_0(Y)^{*(n)}$ for some locally compact space $Y$,
	\item $E=L_\infty(\Omega,\Sigma,\mu)^{*(n)}$ for some $\sigma$-finite measure space $(\Omega,\Sigma,\mu)$,
        \item $E$ is a predual space of the space $C(K)$ for some compact space $K$,
	\item $E=L_1(\Omega,\Sigma,\mu)$ for some $\sigma$-finite measure space $(\Omega,\Sigma,\mu)$,
	\item $E=\mathcal{L}(Y,Z)$ for some infinite-dimensional normed space $Y$ and some Banach space $Z$ containing a copy of $c_0$,
	\item {\normalfont$E=\mbox{Bil}(Y\times W,Z)$} for some infinite-dimensional Banach spaces $Y$, $W$, and $Z$ such that $Z$ contains a copy of $c_0$,
	\item $E=\lip_0(M)^{*(n)}$ for some metric space $M$,
	\item $E=\mathcal{F}(M)$ for some metric space $M$.
\end{enumerate}
\end{theorem}
\begin{proof}
(1) Recall that the space $C_0(Y)$ contains a copy of $c_0$, the dual space $C_0(Y)^*$ contains a complemented copy of $\ell_1$ and hence that the bidual space $C_0(Y)^{**}$ contains a complemented copy of $\ell_\infty$. It follows that, for each odd $m\in\mathbb{N}$, the space $C_0(Y)^{*(m)}$ contains a complemented copy of $\ell_1$ and the space $C_0(Y)^{*(m+1)}$ contains a complemented copy of $\ell_\infty$ and so a copy of $c_0$. Consequently, if $n$ is even, we use Corollary~\ref{cor:bilipschitz_c0}, and if $n$ is odd, we appeal to Corollary~\ref{cor:l1inXnoGrothendieck}.

(2) follows from \cite[Example 4.2.9]{AK} and (1).

(3) Since $C(K)$ contains a copy of $c_0$ and is a dual space, by \cite[Corollary 1.2]{Rosenthal} $E$ contains a complemented copy of $\ell_1$ and so we use Corollary~\ref{cor:l1inXnoGrothendieck}.

(4) follows from \cite[Example 4.2.9]{AK} and (3).

(5) By \cite[Theorem 1]{FerrandoLYZ}, $E$ contains a copy of $c_0$, so we use Corollary~\ref{cor:bilipschitz_c0}.

(6) Note that $E$ is isometrically isomorphic to $\mathcal{L}(Y\widehat{\otimes}_\pi W,Z)$, the space of all continuous  operators from the projective tensor product of $Y$ and $W$ into the space $Z$, and use (5).

(7) If $E=\lip_0(M)^{*(n)}$ and $n$ is even, then, by the discussion after Corollary \ref{separable}, $E$ contains a copy of $\ell_\infty$ and so, in particular, a copy of $c_0$; we then use Corollary~\ref{cor:bilipschitz_c0}. If $n$ is odd, then, similarly, $E$ contains a complemented copy of $\ell_1$, so we use Corollary~\ref{cor:l1inXnoGrothendieck}.

(8) follows from Corollaries~\ref{hano_ell1} and \ref{cor:l1inXnoGrothendieck}.
\end{proof}

Even though in this paper we work exclusively with real Banach spaces, the arguments presented in the proof of Theorem~\ref{theorem:examples_onto_ell1} can be generalized to the class of \emph{complex} C*-algebras. However, while in the below proof we use the fact that the C*-algebras are complex to obtain either $c_0$ or $\ell_1$ as a (complemented) subspace of $E$,
for the space $\lip_0(E)$ (consisting of real-valued functions), we consider $E$ as a real Banach space by forgetting its additional complex structure. In fact, in the cases where we exhibit a copy of $c_0$ in $E$, taking into account Corollary~\ref{cor:bilipschitz_c0}, we use only that $E$ is a metric space containing a bilipschitz copy (of the unit sphere) of $c_0$. Note that the following result in particular applies to the C*-algebras $X=\mathcal{B}(H)$ of all bounded operators on
infinite-dimensional Hilbert spaces $H$.

\begin{theorem}\label{theorem:cstaralgebras}
Let $n\ge0$. Let $E$ be an infinite-dimensional complex Banach space. If $E=X^{*(n)}$ for some C*-algebra $X$ or $E$ is a predual of a C*-algebra, then $\lip_0(E)$ admits a continuous  operator onto $\ell_1$.
\end{theorem}
\begin{proof}
Assume first that $E$ is an infinite-dimensional C*-algebra. Then, it is well-known that $E$ contains an infinite-dimensional maximal abelian subalgebra $M$ (see \cite[Exercise 4.6.12]{KR83}). By the classical Gelfand--Naimark theorem $M$ is isomorphic to the complex space $C_0(X)$ for some infinite locally compact space $X$. Consequently, $M$ contains a copy of the real Banach space $c_0$ and so $E$ contains such a copy, too. By Corollary~\ref{cor:bilipschitz_c0} we get a continuous operator from $\lip_0(E)$ onto $\ell_1$.

Assume that $E=X^{*(n)}$ for some infinite-dimensional C*-algebra $X$ and $n>0$. If $n$ is even, then $E$ carries a compatible C*-algebraic structure (see \cite[Section 1.17]{Sak71}) and so, by the previous argument, $E$ contains a copy of the real Banach space $c_0$. If $n$ is odd, then similarly the complex Banach space $c_0$ can be embedded into complex $X^{*(n+1)}=E^*$ and hence, by \cite[Corollary 1.2]{Rosenthal}, the complex Banach space $E$ contains a complemented copy of the complex Banach space $\ell_1$. Consequently, the real Banach space $E$ contains a complemented copy of the real Banach space $\ell_1$. In either case, referring to Corollary~\ref{cor:bilipschitz_c0} or to Corollary~\ref{cor:l1inXnoGrothendieck}, we get a continuous operator from $\lip_0(E)$ onto $\ell_1$.

If $E$ is a predual of a C*-algebra, then we proceed similarly as in the previous paragraph for odd $n$ to show that $E$ contains a complemented copy of $\ell_1$ and we use Corollary~\ref{cor:l1inXnoGrothendieck}.
\end{proof}

Recall that for a metric space $M$ a function $f\colon M\rightarrow\mathbb{R}$ is called \emph{locally flat} if for every $x\in M$ and every $\varepsilon >0$ there is $\delta>0$ such that $|f(y)-f(z)| \le \varepsilon \rho(y,z)$ for all points~$y,z\in B(x,\delta)$. For a compact metric space $M$ with the base point, the space~$\mathrm{lip}_0(M)$ is defined as the subspace of $\lip_0(M)$ consisting of all locally flat functions. For non-compact metric spaces $M$ additional flatness assumptions are required in the definition of $\mathrm{lip}_0(M)$, see~\cite[Definition~4.15]{Weaver}. The space $\mathrm{lip}_0(M)$ is said to \emph{separate points uniformly} if there is a constant $a\in(0,1]$ such that for all $x,y\in M$ there is a function $f\in\mathrm{lip}_0(M)$ with $\|f\|_{\lip_0(M)}\le 1$ and $|f(x)-f(y)|=a\rho(x,y)$.  The class of all metric spaces $M$ for which the spaces $\mathrm{lip}_0(M)$ separate points uniformly contains e.g. the Cantor space $2^\mathbb{N}$ and all countable compact spaces, see \cite[Section 4.1]{Weaver}.

By \cite[Theorem~4.38]{Weaver}, if a metric space $M$ is infinite and the space $\mathrm{lip}_0(M)$ separates points uniformly, then the dual space of $\mathrm{lip}_0(M)$ is isometrically isomorphic to $\mathcal{F}(M)$, and so, since $\mathcal{F}(M)$ contains a complemented copy of $\ell_1$, by Proposition~\ref{prop:lip0dualell1}, $\lip_0(\mathrm{lip}_0(M))$ admits a continuous operator onto $\ell_1$.

\begin{proposition}
For an infinite metric space $M$, if the space $\mathrm{lip}_0(M)$ separates points uniformly, then there is a continuous linear surjection $T\colon \lip_0(\mathrm{lip}_0(M))\to\ell_1$.
\end{proposition}

Recall that a subspace $\mathcal{N}$ of a metric space $M$ is a \emph{net} if there are $\varepsilon,\delta>0$ such that $\rho(x,y)\ge\varepsilon$ for every $x\neq y\in\mathcal{N}$ and for every $x\in M$ there is $y\in\mathcal{N}$ with $\rho(x,y)<\delta$. Candido, C\'{u}th, and Doucha \cite[Remark 3.6]{C-C-D} observed that the space $\lip_0(\mathcal{N})$, where $\mathcal{N}$ is a net in either $c_0$ or $\ell_1$, admits a continuous operator onto~$\ell_1$ (cf. Remark \ref{rem:ccd}). Of course, no such net $\mathcal{N}$ can contain any bilipschitz copy of the unit sphere of an infinite-dimensional Banach space. The next proposition shows that also for some infinite-dimensional connected complete metric spaces $M$ the existence of the unit sphere of an infinite-dimensional normed space is not necessary for the space $\lip_0(M)$ to admit a continuous operator onto $\ell_1$. Recall that a topological space $X$ is \emph{arcwise-connected} if every two distinct points of $X$ can be connected by a homeomorphic image of the unit interval $[0,1]$.

\begin{proposition}\label{prop:NoGrothendieckNoc0Ball}
  There is an arcwise-connected $\sigma$-compact complete metric space $M$ of infinite covering dimension which does not contain any bilipschitz copy of the unit sphere of an infinite-dimensional normed space but $\lip_0(M)$ still admits a continuous operator onto $\ell_1$.
\end{proposition}
\begin{proof}
  We consider the metric space $M=\bigsqcup_{n\in\mathbb{N}} \ell_\infty^n$ which is defined in accordance with \cite[Definition 1.13]{Weaver} as follows: it is the quotient space of the disjoint union $\bigsqcup_{n\in\mathbb{N}}\ell_\infty^n$ with all the base points of the spaces $\ell_\infty^n$ identified as the new single base point, and for points $x\in \ell_\infty^n$ and $y\in\ell_\infty^{m}$ their distance is defined as
  \[
    \rho(x,y) =\begin{cases}
	 \|x\|_\infty + \|y\|_\infty&,\text{ if }n\neq m,\\
	 \|x-y\|_\infty&,\text{ if }n=m.
	\end{cases}
  \]
  
  Note that $M$ is an arcwise-connected $\sigma$-compact complete metric space. Since for each $n\in\mathbb{N}$ the space $M$ contains $\ell_\infty^n$ as a closed subset, it has infinite covering dimension, see e.g. \cite[Theorem~3.1.4]{Engelking}.
  
  We claim that $M$ contains no homeomorphic copy of the unit sphere of an infinite-dimensional normed space. To see this, note first that the sphere $S$ of an infinite-dimensional normed space has infinite covering dimension and does not contain any point $x$ such that $S\setminus\{x\}$ is disconnected, so every homeomorphic copy of $S$ would have to be contained in one of the $n$-dimensional pieces $\ell_\infty^n$ of $M$. Since homeomorphisms preserve the covering dimension, see e.g. \cite[p.~54]{Engelking}, this would result in a closed infinite-dimensional subset of $\ell_\infty^n$ for some $n$, which is impossible.
  
  By \cite[Proposition 2.8]{Weaver}, we have the isomorphism
  \[
    \lip_0(M) \simeq \Big(\bigoplus_{n\in\mathbb{N}} \lip_0(\ell_\infty^n)\Big)_\infty.
  \]
  The space on the right hand side is isomorphic to $\lip_0(c_0)$ by \cite[Remark 3.2]{C-C-D}, so, by Proposition~\ref{prop:lip0dualell1}, $\lip_0(M)$ admits a continuous operator onto $\ell_1$.
\end{proof}

\subsection{The Grothendieck property of spaces $\lip_0(M)$\label{sec:grothendieck}}

Using results obtained in the first half of Section \ref{sec:ell_1}, we will now provide several sufficient conditions for metric spaces $M$ implying that $\lip_0(M)$ is not a Grothendieck space. Recall that a Banach space $X$ is \emph{Grothendieck} (or has the \emph{Grothendieck property}) if every weak* convergent sequence in the dual space $X^*$ is also weakly convergent. We refer the reader to \cite{Kania} for an extensive survey on Grothendieck Banach spaces.

Typical examples of Grothendieck spaces are reflexive spaces, the space $\ell_\infty$, or, more generally, spaces $C(K)$ for $K$ extremally disconnected. On the other hand, spaces $C(K)$ for $K$ metric are never Grothendieck. R\"{a}biger \cite[Theorem 2.2]{Rab} proved that a Banach space $X$ is a Grothendieck space if and only if $X^*$ is weakly sequentially complete and $X$ has no quotient isomorphic to $c_0$. Since for every compact space $K$ the dual space $C(K)^*$ is always weakly sequentially complete, R\"{a}biger's theorem, with an aid of the results of Cembranos \cite{Cem84} and Schachermayer \cite{Sch82}, implies that a space $C(K)$ is not Grothendieck if and only if $C(K)$ has a quotient isomorphic to $c_0$ if and only if $C(K)$ contains a complemented copy of $c_0$. For spaces $\lip_0(M)$ the situation is different---each space $\lip_0(M)$ is isometrically isomorphic to the dual Banach space $\mathcal{F}(M)^*$, hence it cannot contain any complemented copy of $c_0$ (see e.g. \cite[Theorem 2.4.15]{DDLS}), even though we know by Theorem~\ref{cor1} that it contains \emph{some} copy of $c_{0}$. Consequently, to prove that a given space $\lip_0(M)$ is not Grothendieck, using R\"{a}biger's criterion, we really need either to look for a quotient of $\lip_0(M)$ isomorphic to $c_0$ or to show that the dual space $\lip_0(M)^*$ is not weakly sequentially complete. To achieve the former, the results obtained in the previous section together with the standard fact that $\ell_1$ has a quotient isomorphic to $c_0$ might be useful.

Since the Grothendieck property is preserved by continuous linear surjections, we immediately obtain that $\lip_0(E)$ does not have the Grothendieck property whenever~$E^*$ fails to have it (as, recall, $\lip_0(E)$ contains a complemented copy of $E^*$). We thus immediately get the following examples of non-Grothendieck spaces $\lip_0(E)$ with $E$ Banach. 
For the definitions of properties $(V)$ and $(V^*)$, see e.g. \cite{HWW}.

\begin{proposition}\label{prop:dual_not_groth}
  Let $E$ be a Banach space satisfying any the following assumptions:
  \begin{enumerate}
  \item $E$ has property~$(V)$ and is not Grothendieck,
  \item $E$ is non-reflexive and $E^*$ has property $(V^*)$,
  \item $E$ is non-reflexive and $E^*$ is separable.
  \end{enumerate}
  Then, $\lip_0(E)$ is not a Grothendieck space.
\end{proposition}

\begin{proof}
  Since $\lip_0(E)$ contains a complemented copy of $E^*$ and the Grothendieck property is inherited by complemented subspaces, we only have to check that in each of the above cases $E^*$ is not Grothendieck.

(1) Recall that a Banach space with property (V) is a Grothendieck space if and only if it does not contain a complemented copy of $c_0$, see e.g.~\cite[Proposition~3.1.13]{Kania}. It follows that $E$ contains a complemented copy of $c_0$ and hence $E^*$ contains a complemented copy of $\ell_1$. This implies that $E^*$ is not Grothendieck.

(2) As $E$ is non-reflexive, $E^*$ is non-reflexive, too. By~\cite[Corollary~3.3.C]{HWW}, a non-reflexive space with property $(V^*)$ contains a complemented copy of $\ell_1$ and hence is not Grothendieck.

(3) $E^*$ is non-reflexive, because $E$ is non-reflexive. Since a separable Banach space is a Grothendieck space if and only if it is reflexive, $E^*$ is not Grothendieck.
\end{proof}

\begin{remark}\label{Re}
The space $\lip_0(M)$ for infinite $M$ has never property $(V^*)$, since, by \cite[Proposition~3.3.E]{HWW}, property $(V^*)$ is inherited by subspaces and, by Theorem~\ref{thm:linfLip}, $\lip_0(M)$ contains a copy of $\ell_\infty$, which does not have property $(V^*)$ (since non-reflexive spaces with property $(V^*)$ contain complemented copies of $\ell_1$, by \cite[Corollary~3.3.C]{HWW}).

Since every dual space with property $(V)$ has the Grothendieck property, the examples considered above of spaces $\lip_0(M)$ without the Grothendieck property are also examples of spaces $\lip_0(M)$ without property $(V)$.
\end{remark}

\begin{remark}\label{rem:linden}
Note that in the first two cases in Proposition~\ref{prop:dual_not_groth} we have chains of continuous linear surjections $\lip_0(E)\rightarrow E^*\rightarrow\ell_1\rightarrow c_0$, so in fact these two cases extend the list provided in Theorem~\ref{theorem:examples_onto_ell1}. In the third case we obtain \emph{a priori} only a chain of continuous operators $\lip_0(E)\xrightarrow{T} E^*\xrightarrow{S} c_0$, where $T$ is a projection and $S$ is non-weakly compact; whether there exists a continuous linear surjection $\lip_0(E)\rightarrow\ell_1$ is not clear in this case. It might happen that there is no such surjection but there still exists a continuous linear surjection $\lip_0(E)\rightarrow c_0$, not only a non-weakly compact operator. A possible example could be the space $E=Z^*$, where $Z$ is the separable Banach space constructed by Lindenstrauss in \cite{Lindjames} for $X=c_0$. Note that the third dual space $Z^{***}$ is also separable and hence there cannot exist a continuous linear surjection from the (bi)dual space $E^*=Z^{**}$ onto $\ell_1$, yet by the construction there exists a continuous linear surjection $E^*=Z^{**}\rightarrow c_0$.
\end{remark}

Results obtained in the first part of this section may be quickly used to obtain further sufficient conditions for a Banach space $E$ implying that the space $\lip_0(E)$ is not Grothendieck, thus completing Proposition~\ref{prop:dual_not_groth}.

\begin{theorem}
Let $E$ be a Banach space satisfying any of the following conditions:
\begin{enumerate}
	\item there is a continuous linear surjection $T\colon E^*\to\ell_1$,
	\item $E$ is separable and contains an isomorphic copy of a predual of $\ell_1$,
	\item $E$ contains a complemented copy of $\ell_1$.
\end{enumerate}
Then, $\lip_0(E)$ is not a Grothendieck space.
\end{theorem}
\begin{proof}
(1) follows from Proposition~\ref{prop:lip0dualell1}, (2) from Lemma~\ref{lemma:Lip0CKNotGP}, and (3) from Corollary~\ref{cor:l1inXnoGrothendieck}.
\end{proof}

\begin{corollary}
If $E=c_0$ or $E=\ell_1$, then $\lip_0(E)$ is not a Grothendieck space.
\end{corollary}

Theorem~\ref{thm:BiLipCK} immediately yields the following results.

\begin{theorem}\label{Fo}
Let $E$ be a separable Banach space which is an absolute Lipschitz retract and contains an isomorphic copy of $c_0$. If a metric space $M$ contains a bilipschitz copy of the unit sphere $S_{E}$ of $E$, then the space $\lip_0(M)$ is not a Grothendieck space.
\end{theorem}

\begin{corollary}\label{Gro}
If a metric space $M$ contains a bilipschitz copy of the unit sphere $S_{c_0}$ of $c_0$, then $\lip_0(M)$ is not a Grothendieck space.
\end{corollary}

In Theorem~\ref{theorem:examples_onto_ell1} we have provided an extensive list of examples of Banach spaces $E$ such that the spaces $\lip_0(E)$ admit continuous operators onto $\ell_1$. Consequently, these spaces $\lip_0(E)$ do not have the Grothendieck property.

\begin{corollary}\label{cor:many_non_grothendieck_spaces}
If $E$ is an infinite-dimensional Banach space satisfying any of conditions (1)--(8) of Theorem~\ref{theorem:examples_onto_ell1}, then $\lip_0(E)$ is not a Grothendieck space.

In particular, if $E$ is a $C(K)$-space, $L_1(\mu)$-space, $\lip_0(M)$-space, or $\mathcal{F}(M)$-space, then $\lip_0(E)$ is not Grothendieck.
\end{corollary}

Proposition~\ref{prop:Retract1}, which was the main ingredient of the proof of Theorem~\ref{thm:BiLipCK} and hence also of several of the above results, has an immediate non-Grothendieck analogon as well.

\begin{proposition}\label{prop:Retract Grothendieck}
Let $N$ be a metric space and $M\subseteq N$. If $M$ is a Lipschitz retract of $N$, then:
\begin{enumerate}
  \item if $\lip_0(N)$ is weakly sequentially complete, so is $\lip_0(M)$,
  \item if there is a continuous linear surjection $\lip_0(M)\rightarrow c_0$, then there is one from $\lip_0(N)$ onto $c_0$, and
  \item if $\lip_0(N)$ has the Grothendieck property, so does $\lip_0(M)$.
  \end{enumerate}
  In particular, if $M$ is an absolute Lipschitz retract such that $\lip_0(M)$ does not have the Grothendieck property (e.g. $M=c_0$), then for every metric space $N$ containing a bilipschitz copy of $M$ the space $\lip_0(N)$ does not have the Grothendieck property.
\end{proposition}

By \cite[Theorem 2.2]{Rab}, if the dual $E^*$ of a Banach space $E$ is not weakly sequentially complete, then $E$ is not a Grothendieck space. Using this theorem and the lifting property of separable Banach spaces, we make the following observation (cf. \cite[Remark 3.6]{C-C-D}), which can be applied to the James space $\mathcal{J}$.

\begin{proposition}\label{prop:wsc}
  Let $E$ be a separable Banach space which is not weakly sequentially complete. Then, $\lip_0(E)$ is not a Grothendieck space.
\end{proposition}
\begin{proof}
Since separable spaces have the lifting property (see \cite[Theorem~3.1]{GK}), the Lipschitz-free space $\mathcal{F}(E)$ contains $E$ as a closed linear subspace (by \cite[Theorem 2.12]{GK}), and hence $\lip_0(E)^*\simeq\mathcal{F}(E)^{**}$ cannot be weakly sequentially complete. Consequently, $\lip_0(E)$ does not have the Grothendieck property.
\end{proof}

Recall that the James space $\mathcal{J}$ is an example of a non-reflexive quasi-reflexive separable infinite-dimensional Banach space (see \cite[Section 3.4]{AK}).

\begin{corollary}\label{cor:james}
$\lip_0(\mathcal{J})$ is not Grothendieck.
\end{corollary}
\begin{proof}
The statement can be justified twofold. First, by the quasi-reflexivity, the dual space $\mathcal{J}^*$ is also separable, hence, by Proposition \ref{prop:dual_not_groth}, $\lip_0(\mathcal{J})$ is not Grothendieck. Second, again by the quasi-reflexivity, $\mathcal{J}$ does not contain any copies of $c_0$ or $\ell_1$. Consequently, by the non-reflexivity and Rosenthal's $\ell_1$-theorem, $\mathcal{J}$ is not weakly sequentially complete. Thus, Proposition \ref{prop:wsc} applies to $\mathcal{J}$, too.
\end{proof}

\begin{remark}\label{rem:ccd}
Let us mention that the argument presented in the proof of Proposition \ref{prop:wsc} was first used in \cite[Remark 3.6]{C-C-D} to show that for any net $\mathcal{N}$ in the Banach space $c_0$ the space $\lip_0(\mathcal{N})$ is not Grothendieck.
\end{remark}

Summing up the above results, we have the following cases of metric space $M$ for which the space $\lip_0(M)$ does not have the Grothendieck property:
\begin{enumerate}
\item $M$ is a separable Banach space which is not weakly sequentially complete, e.g. $M$ is the James space $\mathcal{J}$;
\item $M$ contains a bilipschitz copy of the unit sphere $S_{c_0}$, e.g. $M=C(K)$ for some infinite compact space $K$;
\item $M=\bigsqcup_{n\in\mathbb{N}} \ell_\infty^n$;
\item $M$ is a net in $c_0$ or $\ell_1$.
\end{enumerate}

An infinite dimensional Banach space for which the space of Lipschitz functions does not have the Grothendieck property, but neither (1) nor (2) above is satisfied, is~$\ell_1$. By \cite[Proposition~79]{Dalet2} the space $\ell_1$ is complemented in~$\lip_0(\ell_1)$ and hence the latter space cannot have the Grothendieck property. Since $\ell_1$ has the Schur property, it is weakly sequentially complete, so (1) is not satisfied. In order to prove that (2) is not satisfied as well, we use that by the corollary to ~\cite[Theorem~5.1]{Raynaud} there is no uniform and hence no bilipschitz embedding of $S_{c_0}$ into a stable Banach space and that a direct consequence of \cite[Theorem~0.1]{Raynaud} is that $\ell_1$ is stable\footnote{The authors would like to thank Bunyamin Sari for pointing out Raynaud's paper to them.}.

\section{Open problems and final remarks\label{sec:problems}}

In this final section of our paper we provide several open questions concerning the existence of continuous surjections from spaces of Lipschitz functions.

We start with the following problem related to Theorem~\ref{theorem_lipp_onto_lipw} and Proposition~\ref{prop_lipp_seq_onto_lipw}.

\begin{problem}
  Assume that  metric spaces $M$ and $N$ have the same density.
  Are the spaces $\lip_0(M)_w$ and   $\lip_0(N)_w$  homeomorphic?
\end{problem}

Theorem~\ref{theorem_cpm_homeo_fmw} provides several criteria implying that, for given infinite metric spaces $M$ and $N$, the spaces $C_p(M)$ and $\mathcal{F}(N)_w$ are not homeomorphic. We are, however, not aware of any examples of spaces $M$ and $N$ such that $C_p(M)$ and $\mathcal{F}(M)_w$ \emph{are} homeomorphic. Conditions (1)--(6) in Theorem~\ref{theorem_cpm_homeo_fmw} imply that if such spaces $M$ and $N$ do exist, then they both must be separable, $M$ must be non-scattered $\sigma$-compact and of cardinality $2^{\aleph_0}$, and $\mathcal{F}(N)$ cannot have the Schur property (which, by \cite[Theorem C]{AGPP}, is equivalent for the completion of $N$ to contain a bilipschitz copy of a subset of $\mathbb{R}$ of non-zero Lebesgue measure), nor it can be embeddable into a separable $L_1(\mu)$-space. Natural candidates could be then, e.g., $M=N=\mathbb{R}^2$.

\begin{problem}\label{LL}
  Do there exist infinite metric spaces $M$ and $N$ such that the spaces $C_{p}(M)$ and  $\mathcal{F}(N)_w$ are homeomorphic?
\end{problem}

In Section~\ref{sec:ell_1} we studied for which infinite metric spaces $M$ their spaces $\lip_0(M)$ admit continuous linear surjections onto the space $\ell_1$ or the space $c_0$. It is well-known that for $M=\mathbb{R}$ such mappings do not exist, as $\lip_0(M)\simeq\ell_\infty$ in this case. Unfortunately, we do not know any examples of Banach spaces $E$ of dimension at least $2$ for which such operators from $\lip_0(E)$ do not exist (cf. also Remark \ref{rem:linden}).

\begin{problem}\label{c_0quotient}
Does there exist a Banach space $E$ of dimension at least $2$ and such that the space $\lip_0(E)$ admits no continuous linear  surjection onto $\ell_1$? Can such $\lip_0(E)$ still admit a continuous linear surjection onto $c_0$?
\end{problem}

We also do not know any example of an infinite-dimensional Banach space $E$ such that the space $\lip_0(E)$ has the Grothendieck property.

\begin{problem}\label{banach_gr_problem}
  Is there an infinite-dimensional Banach space $E$ for which the space $\lip_0(E)$ is a Grothendieck space?
\end{problem}

Note that a positive answer to Problem \ref{banach_gr_problem} would also give a positive answer to the first question in Problem \ref{c_0quotient}, as the Grothendieck property is preserved by quotients.

Natural variants of Problem~\ref{banach_gr_problem} are: 1) Is there a reflexive Banach space $X$ such that $\lip_0(X)$ is Grothendieck? 2) Is the space $\lip_0(\ell_2)$ Grothendieck? 3) Is the space $\lip_0(\mathbb{R}^2)$ Grothendieck? Let us discuss briefly the latter two questions as both positive and negative answers to them would have several interesting consequences.
\begin{itemize}
   \item If $\lip_0(\ell_2)$ is a Grothendieck space, then for any $d\in\mathbb{N}$ the space $\lip_0(\mathbb{R}^d)$ is Grothendieck as well, since $\mathbb{R}^d$ is a Lipschitz retract of $\ell_2$. This would answer \cite[Question 9]{C-C-D}.
   \item If $\lip_0(\ell_2)$ is a Grothendieck space, then we have an interesting situation, since by Corollary \ref{cor:james} for the James space $\mathcal{J}$ the space $\lip_0(\mathcal{J})$ is not Grothendieck and it is known that $\mathcal{J}$ is $\ell_2$-saturated, that is, each of its closed infinite-dimensional linear subspaces contains an isomorphic copy of the space $\ell_2$ (see \cite[Lemma 1]{HW67}).
   \item If $\lip_0(\ell_2)$ is not Grothendieck, then for any separable Banach space $E$ with non-trivial type (e.g. $E=L_p([0,1]$) for $1<p<\infty$) the space $\lip_0(E)$ is not Grothendieck. This follows from \cite[Theorem 3.1]{C-C-D} which asserts that for such $E$ the space $\lip_0(E)$ contains a complemented copy of $\lip_0(\ell_2)$.
  \item If $\lip_0(\ell_2)$ is not Grothendieck and one can find a separable infinite-dimensional Banach space $E$ such that $\lip_0(E)$ is Grothendieck, then this would answer in negative \cite[Question~7]{C-C-D}, asking whether $\lip_0(\ell_2)$ is complemented in $\lip_0(F)$ for every separable infinite-dimensional Banach space $F$.
  \item Recall that the space $\lip_0(\mathbb{R}^2)$ is isomorphic to the space $\lip_0([0,1]^2)$ (see \cite[Corollary 3.5]{Kaufmann}), where the unit square $[0,1]^2$ carries the metric induced by the maximum norm on $\mathbb{R}^2$ and is an absolute Lipschitz retract. Consequently, if $\lip_0(\mathbb{R}^2)$ is not Grothendieck, then for every metric space $M$ containing a bilipschitz copy of $[0,1]^2$ the space $\lip_0(M)$ does not have the Grothendieck property.
 \end{itemize}

\begin{remark}
For any $n\ge1$, the space $\lip_0(\mathbb{R}^{n})^*$ cannot have the Grothendieck property, since, by \cite[Corollary~1.2]{C-K-K}, it contains $\mathcal{F}(\mathbb{R}^{n})$ as a complemented subspace and hence, by Corollary \ref{hano_ell1}, also a complemented copy of $\ell_1$.
\end{remark}

We finish the paper with the following important and well-known question (cf. \cite[Introduction]{LeandroGuzman}). Note that a positive answer to Problem~\ref{c_0quotient} or Problem~\ref{banach_gr_problem} would provide a positive answer to Problem~\ref{non_isomorph_lip0} as well.

\begin{problem}\label{non_isomorph_lip0}
Do there exist two infinite-dimensional Banach spaces $E$ and $F$ having the same density and such that $\lip_0(E)$ and $\lip_0(F)$ are not isomorphic?
\end{problem}

By the fact that for every Banach space $E$ the space $\lip_0(E)$ contains a complemented copy of the dual space $E^*$, if for every two infinite-dimensional Banach spaces $E$ and $F$ of the same density the spaces $\lip_0(E)$ and $\lip_0(F)$ were isomorphic, then we would have the following interesting situation: given an infinite-dimensional Banach space $E$ of density $\kappa$, the space $\lip_0(E)$ would contain a complemented copy of the dual space $F^*$ for every infinite-dimensional Banach space $F$ of density $\le\kappa$. In particular, the space $\lip_0(c_0)$ would contain a complemented copy of the dual space $F^*$ for every separable Banach space $F$, which would provide a simple example of a Banach space with such a property and a separable predual. Let us note that such Banach spaces have been already constructed, though they are non-trivial, see  e.g. \cite[Theorem 1]{Johnson}.

\section*{Acknowledgements}
The authors wish to thank the anonymous referee for carefully reading the paper and for a number of valuable suggestions.

The first and the third named authors were supported by the Austrian Science Fund (FWF):~I~4570-N. The third named author was also supported by the Austrian Science Fund (FWF):~ESP~108-N.

This research was funded in whole or in part by the Austrian Science Fund (FWF), grant DOI: 10.55776/ESP108, 10.55776/I4570. For open access purposes, the authors have applied a CC BY public copyright license to any author-accepted manuscript version arising from this submission.

\section{Statements and Declarations}

\noindent\textbf{Funding.} The first and the third named authors were supported by the Austrian Science Fund~(FWF): I~4570-N. The third named author was also supported by the Austrian Science Fund~(FWF): ESP~108-N. 

This research was funded in whole or in part by the Austrian Science Fund (FWF), grant DOIs: 10.55776/ESP108, 10.55776/I4570. For open access purposes, the authors have applied a CC BY public copyright license to any author-accepted manuscript version arising from this submission.

\noindent\textbf{Competing Interests.} The authors have no relevant financial or non-financial interests to disclose.

\noindent\textbf{Author Contributions.} All authors contributed equally and to all parts of the article.

\noindent\textbf{Data availability statement.} Data sharing not applicable to this article as no datasets were generated or analysed during the current study.

\end{document}